\documentclass[11pt]{amsart}

\usepackage[a4paper,hmargin=3.5cm,vmargin=4cm]{geometry}
\usepackage{amsfonts,amssymb,amscd,amstext}
\usepackage{graphicx}
\usepackage[dvips]{epsfig}
\usepackage{quoting}
\usepackage{hyperref}  

\usepackage{fancyhdr}
\usepackage{enumerate}
\usepackage{titlesec}
\usepackage{mathrsfs}

\titleformat{\section}
{\filcenter\bfseries\large} {\thesection{.}}{0.2cm}{}
\titleformat{\subsection}[runin]
{\bfseries} {\thesubsection{.}}{0.15cm}{}[.]
\titleformat{\subsubsection}[runin]
{\em}{\thesubsubsection{.}}{0.15cm}{}[.]

\usepackage[up,bf]{caption}

\newtheorem{theorem}{Theorem}[section]
\newtheorem{proposition}[theorem]{Proposition}
\newtheorem{lemma}[theorem]{Lemma}
\newtheorem{claim}[theorem]{Claim}

\theoremstyle{definition}
\newtheorem{definition}[theorem]{Definition}
\newtheorem{remark}[theorem]{Remark}

\numberwithin{equation}{section}
\numberwithin{figure}{section}

\newcommand\Acal{\mathcal{A}}

\newcommand\Ccal{\mathcal{C}}

\newcommand\Ocal{\mathcal{O}}
\newcommand\Pcal{\mathcal{P}}

\newcommand\Scal{\mathcal{S}}

\newcommand\C{\mathbb{C}}

\renewcommand\P{\mathbb{P}}

\newcommand\R{\mathbb{R}}

\renewcommand\b{\mathbb{B}}
\renewcommand\c{\mathbb{C}}

\newcommand\cp{\mathbb{CP}}
\renewcommand\d{\mathbb D}

\newcommand\n{\mathbb{N}}
\renewcommand\r{\mathbb{R}}
\newcommand\s{\mathbb{S}}

\newcommand\z{\mathbb{Z}}

\newcommand\igot{\mathfrak{i}}

\renewcommand\igot{\mathfrak{i}}

\newcommand\pgot{\mathfrak{p}}

\newcommand\sgot{\mathfrak{s}}

\newcommand\Agot{\mathfrak{A}}

\newcommand\Sgot{\mathfrak{S}}

\renewcommand\imath{\igot}

\newcommand\wt{\widetilde}

\newcommand\di{\partial}

\newcommand\dist{\mathrm{dist}}

\newcommand\Div{\mathrm{Div}}

\newcommand\Flux{\mathrm{Flux}}
\newcommand\GCMI{\mathrm{GCCMI}}
\newcommand\CMI{\mathrm{CCMI}}

\newcommand\supp{\mathrm{supp}}
\newcommand\Qcal{\mathcal{Q}}

\def\dist{\mathrm{dist}}

\def\Flux{\mathrm{Flux}}

\usepackage{color}

\begin{document}


\fancyhead[LO]{Algebraic approximation and the Mittag-Leffler theorem}
\fancyhead[RE]{A.\ Alarc\'on and F.J. L\'opez}
\fancyhead[RO,LE]{\thepage}

\thispagestyle{empty}



\begin{center}
{\bf\LARGE Algebraic approximation and the Mittag-Leffler theorem for minimal surfaces}

\vspace*{5mm}

%
%
{\large\bf Antonio Alarc\'on\quad and\quad Francisco J. L\'opez}
\end{center}

%
%
\vspace*{5mm}

\begin{quote}
{\small
\noindent {\bf Abstract}\hspace*{0.1cm} 
In this paper, we prove a uniform approximation theorem with interpolation for complete conformal minimal surfaces with finite total curvature in the Euclidean space $\r^n$ $(n\ge 3)$. 
As application, we obtain a Mittag-Leffler type theorem for complete conformal minimal immersions $M\to\r^n$ on any open Riemann surface $M$.

\smallskip

\noindent{\bf Keywords}\hspace*{0.1cm} 
Minimal surface, Riemann surface, meromorphic function.

\smallskip

\noindent{\bf Mathematics Subject Classification (2010)}\hspace*{0.1cm} 53A10, 53C42, 30D30, 32E30.

}
\end{quote}



\section{Introduction}\label{sec:intro}

Holomorphic approximation and interpolation is a fundamental subject in complex analysis that plays an important role in several fields of Mathematics. In particular, it has been fundamental by means of the classical Enneper-Weiertrass representation formula in the development of the theories of approximation and interpolation for conformal minimal surfaces in the Euclidean space $\r^n$; a classical field of geometry. We refer to \cite{AlarconForstneric2019JAMS} for a survey of recent results in this subject.  

We start by recalling the following two seminal theorems in complex analysis from the late 19th Century.
\begin{enumerate}[A]

\item[$\bullet$] Runge's theorem (1885): if $K\subset \c$ is a compact set such that $\c\setminus K$ is connected, then every holomorphic function on a neighborhood of $K$ can be approximated uniformly on $K$ by entire polynomials \cite{Runge1885AM}. Mergelyan's theorem from 1951 ensures that it suffices to ask the function to be continuous on $K$ and holomorphic on the interior $\mathring K$ of $K$ \cite{Mergelyan1951DAN}.

\smallskip
\item[$\bullet$] Mittag-Leffler's theorem (1884): if $A\subset\c$ is a closed discrete subset and if $f$ is a meromorphic function on a neighborhood of $A$, then there is a meromorphic function $\tilde f$ on $\c$ such that $\tilde f$ is holomorphic on $\c\setminus A$ and $\tilde f-f$ is holomorphic at every point of $A$ \cite{Mittag-Leffler1884AM}. This a sort of dual to the Weierstrass theorem from 1876 ensuring that for any 
map $r:A\to\n$ there is an entire function having a zero of order $r(a)$ at each point $a\in A$ and vanishing nowhere else \cite{Weierstrass1876}. 
\end{enumerate}

The aim of this paper is to provide analogues of these results in the global theory of minimal surfaces in $\r^n$ for $n\ge 3$ (see Theorems \ref{th:intro1} and \ref{th:intro2}).

The aforementioned theorems admit several generalizations in complex analysis and algebraic geometry; we refer to the survey of Forn\ae ss, Forstneri\v c, and Wold \cite{FornaessForstnericWold2018} for a review of this classical but still very active 
subject.  
Concerning meromorphic functions on compact Riemann surfaces, we recall the following extension of Runge's theorem, including interpolation, which dates back to the early decades of modern Riemann surface theory (in this paper we use multiplicative notation for divisors). 
 %
 %
\begin{theorem}[Behnke-Stein \cite{BehnkeStein1949}, Royden \cite{Royden1967JAM}]\label{th:BSR}
Let $E$ be a nonempty finite set in a compact Riemann surface $\Sigma$. If $K\subset\Sigma\setminus E$ is a Runge compact subset\footnote{A compact subset $K$ of an open Riemann surface $M$ is said to be {\em Runge} (or {\em holomorphically convex}) in $M$ 
if the complement $M\setminus K$ has no relatively compact connected components.}, if $f$ is a meromorphic function on a neighborhood of $K$, and if $D$ is a finite divisor with the support in $K$, then for any $\epsilon>0$ there is a meromorphic function $\wt f$ on $\Sigma$ such that $\wt f$ is holomorphic on $\Sigma\setminus E$ except for the poles of $f$ in $K$, 
$|\wt f-f|<\epsilon$ on $K$, and the divisor of $\wt f-f$ is a multiple of $D$ in a neighborhood of $K$.
\end{theorem}
%
%

The natural counterpart of meromorphic functions in minimal surface theory are complete minimal surfaces with finite total curvature (we refer e.g.\ to \cite{ChernOsserman1967JAM,Osserman1986,BarbosaColares1986LNM,Yang1994MA} for background on these surfaces). Indeed, if $X\colon M\to \r^n$  is a complete conformal minimal immersion with finite total curvature from an open Riemann surface $M$, then $M$ is biholomorphic to $\Sigma\setminus E$ where $\Sigma$ and $E$ are as in Theorem \ref{th:BSR}. Moreover, the exterior derivative $d X$ of $X\colon\Sigma\setminus E\to\r^n$ (which coincides with its $(1,0)$-part $\di X$ since $X$ is harmonic) is holomorphic and extends meromorphically to $\Sigma$ with an effective pole at each point of $E$ (see \cite{Huber1957CMH,ChernOsserman1967JAM} or \cite{Osserman1986}). These surfaces are since the early works by Osserman in the 1960s a major focus of interest in the global theory of minimal surfaces.

%
%
The following analogue for conformal minimal surfaces in $\R^n$ $(n\ge 3)$  of the Behnke-Stein-Royden theorem is a simplified version of our main result (see Theorem \ref{th:MT} for a more precise statement including Mergelyan approximation and control of the flux).
\begin{theorem}[Runge's theorem for complete minimal surfaces with finite total curvature]\label{th:intro1} 
Let $\Sigma$ be a compact Riemann surface and $\varnothing\neq E\subset \Sigma$ be a finite subset. Also let $K\subset \Sigma\setminus E$ be a smoothly bounded, Runge compact domain  and let $E_0$ and $\Lambda$ be a  pair of disjoint (possibly empty)  finite  sets in $\mathring K$.
If $X\colon K\setminus E_0\to \r^n$ $(n\geq 3)$ is a complete conformal minimal immersion with finite total curvature, then for any $\epsilon>0$ and any integer $r\ge 0$ there is a conformal minimal immersion $Y\colon \Sigma\setminus (E\cup E_0)\to \r^n$ satisfying the following conditions.  
\begin{enumerate}[\rm (i)]
\item $Y$ is complete and has finite total curvature.
\smallskip
\item $Y-X$ extends  harmonically to $K$ and  $|Y-X|<\epsilon$ on $K$. 
\smallskip
\item $Y-X$ vanishes at least to order $r$ at every point of $E_0\cup \Lambda$.
\end{enumerate}
\end{theorem}
Since $X\colon K\setminus E_0\to\r^n$ is complete and has finite total curvature, it is a proper map (see \cite{JorgeMeeks1983T}), and hence $\lim_{p\to E_0}|X(p)|=+\infty$. Likewise, $Y\colon \Sigma\setminus (E\cup E_0)\to \r^n$ is also proper by condition {\rm (i)}; we emphasize that, in view of {\rm (ii)} and {\rm (iii)}, we have that $\lim_{p\to E_0}|Y(p)-X(p)|=0$. 

Theorem \ref{th:intro1} is known in the particular case when $n=3$ and either $\Lambda=\varnothing$ (see \cite{Lopez2014TAMS}) or $E_0=\varnothing$  (see \cite{AlarconCastro-InfantesLopez2019CVPDE}). The methods in \cite{Lopez2014TAMS,AlarconCastro-InfantesLopez2019CVPDE} rely strongly on the spinor representation formula for minimal surfaces in $\r^3$, a tool that is no longer available in higher dimensions. We point out that our proof works in arbitrary dimension, also for $n=3$. Theorem \ref{th:intro1} is the first known approximation or interpolation result by complete minimal surfaces with finite total curvature in $\R^n$ for $n>3$. 

%

We now pass to consider holomorphic functions on arbitrary open Riemann surfaces. In 1948 Florack \cite{Florack1948}, by building on the methods developed by Behnke and Stein in \cite{BehnkeStein1949}, provided analogues to the Mittag-Leffler and the Weierstrass theorems in this more general framework.  Likewise, in 1958 Bishop \cite{Bishop1958PJM} extended the Runge-Mergelyan theorem to any open Riemann surface; in this case the approximation takes place in Runge compact subsets 
(see \cite[Theorems 3.8.1 and 5.4.4]{Forstneric2017} for a general statement including interpolation). 
In this direction and as application of Theorem \ref{th:intro1}, we also obtain in this paper the following analogue for minimal surfaces of the Mittag-Leffler theorem that also includes approximation of Runge-Bishop type with interpolation (see Theorem \ref{th:M-Lgen} for a more precise statement).
%
%
\begin{theorem}[Mittag-Leffler's theorem for minimal surfaces]\label{th:intro2}
Let $M$ be an open Riemann surface, $A\subset M$ be a closed discrete subset, and $U\subset M$ be a locally connected, smoothly bounded closed neighborhood of $A$ whose connected components are all Runge compact sets. If $X\colon U\setminus A\to\r^n$ $(n\ge 3)$ is a complete conformal minimal immersion whose restriction to each connected component of $U\setminus A$ has finite total curvature, then there exists a complete conformal minimal immersion $Y\colon M\setminus A\to\r^n$ such that the map $Y-X$ is harmonic at every point of $A$.

Furthermore, given $\epsilon>0$, a closed discrete subset $\Lambda$ of $M$ with $\Lambda\subset \mathring U\setminus  A$, and a map $r:A\cup \Lambda\to\n$, the immersion $Y$ can be chosen such that $|Y-X|<\epsilon$ on $U$ and $Y-X$ vanishes at least to order $r(p)$ at each point $p\in A\cup\Lambda$.
\end{theorem}
%
%
The assumption that $U$ is locally connected is clearly necessary for the last statement in the theorem concerning approximation and interpolation.

In the particular case when $A=\varnothing$, Theorem \ref{th:intro2} is an analogue of the aforementioned Runge-Bishop theorem with jet interpolation and follows easily from the results in \cite{AlarconCastro-Infantes2019APDE}; see also \cite{AlarconLopez2012JDG,AlarconForstneric2014IM,AlarconForstnericLopez2016MZ}. The methods in these sources rely strongly on power complex analytic tools coming from modern Oka theory (we refer to Forstneri\v c \cite{Forstneric2017} for a comprehensive monograph on the subject). On the other hand, a similar result to Theorem \ref{th:intro2} in case $n=3$ and $\Lambda=\varnothing$ was obtained in \cite{Lopez2014JGA}, again using the spinor representation formula for minimal surfaces which is only available in $\r^3$. Theorem \ref{th:intro2} is the first known result of its kind for $A\neq\varnothing$ and $n>3$, even without asking $Y$ to be complete.

As has been made apparent in this introduction, the results we provide in this paper subsume most of the currently known results in the theories of approximation and interpolation for conformal minimal surfaces in $\r^n$, including the somehow simpler case $n=3$. At this time we do not know, for instance, whether the immersion $Y$ in Theorem \ref{th:intro1} can be chosen to be an embedding when $n\ge 5$ and $X|_\Lambda$ is injective; the corresponding result for general minimal surfaces (without taking care of the total curvature) was obtained in \cite{AlarconForstnericLopez2016MZ,AlarconCastro-Infantes2019APDE}. On the other hand, our method of proof also works for null holomorphic curves in the complex Euclidean space $\c^n$ $(n\ge 3)$, and hence the analogous results for these objects of Theorems \ref{th:intro1} and \ref{th:intro2} hold true. It does not seem to us, however, that the approach in this paper could be adapted to deal with more general families of directed holomorphic immersions of open Riemann surfaces as those in \cite{AlarconForstneric2014IM,AlarconCastro-Infantes2019APDE}; nevertheless, we expect that it could be useful to study some particular instances having good algebraic properties.

%
%
\subsection*{Method of proof} 
The proof of Theorem \ref{th:intro1} follows the standard approach of controlling the periods of the Weierstrass data, 
but it presents important innovations. We begin by proving in Section \ref{sec:Mergelyan} a (local) Mergelyan theorem for complete minimal surfaces with finite total curvature (see Theorem \ref{th:Merg}). For that, we adapt the techniques in  \cite{AlarconForstneric2014IM,AlarconForstnericLopez2016MZ}, using the ellipticity of the null quadric $\Agot_*^{n-1}$ of $\c^n$ (see \eqref{eq:nullquadric}) and sprays generated by the flows of complete vector fields along it, which have been developed in the compact case (i.e., when $E_0=\varnothing$). This step is not required if the set $K$ in Theorem \ref{th:intro1} is a strong deformation retract of $\Sigma\setminus E$.

Next, we obtain in Section \ref{sec:Royden} an extension of the Behnke-Stein-Royden theorem (Theorem \ref{th:BSR}) in which extra control on the divisor of the approximating function is provided; see Proposition \ref{pro:Royden0}. This result, which may be of independent interest, is key to ensure the completeness of the immersion $Y$ in Theorem \ref{th:intro1}, as well as to deal with the special case $n=3$. In Section \ref{sec:sprays} we introduce the period dominating sprays that will be used in the proof of the main theorem (see Lemma \ref{le:spray}); the main novelties here are that the sprays are of multiplicative nature and that, instead of working with the null quadric $\Agot_*^{n-1}$, we consider its biholomorphic copy 
\[
	\Sgot_*^{n-1}=\Big\{u=(u_1,\ldots,u_n)\in\c^n\setminus\{0\}\colon  u_1u_2=\sum_{j=3}^n u_j^2\Big\}.
\]
The special geometry of this quadric enables us to approximate, in a simple way, meromorphic maps $u=(u_1,\ldots,u_n)\colon K\to\Sgot_*^{n-1}$, defined on a Runge compact set $K$ of an open Riemann surface $\Sigma\setminus E$ as in Theorem \ref{th:intro1}, by meromorphic maps $\hat u=(\hat u_1,\ldots,\hat u_n)\colon \Sigma\setminus E\to\Sgot_*^{n-1}$. For that, we first approximate $u_1$ by some $\hat u_1$ and then  approximate the $(n-2)$-tuple $(u_3,\ldots,u_n)$ by a suitable meromorphic map $(\hat u_3,\ldots,\hat u_n)$; doing this in the right way, the function $\hat u_2$ defined on $\Sigma\setminus E$ by
\[
	\hat u_2=\frac{\sum_{j=3}^n \hat u_j^2}{\hat u_1}
\]
completes the task.
With the mentioned tools at hand, we prove Theorem \ref{th:MT} (and hence Theorem \ref{th:intro1}) in Section \ref{sec:Runge}. At this point, the main concern is to control the divisors of all the approximating functions at each step of the construction, this enables us to avoid the appearance of branch points and guarantee the completeness of the resulting immersion, while controlling the periods and ensuring the approximation condition. In this stage we shall systematically use the Hurwitz theorem from 1895 (see \cite{Hurwitz1985} or e.g.\  \cite[\textsection VII.2.5, p.\ 148]{Conway1973}) associating the zeros of a convergent sequence of holomorphic functions with the ones of its limit function.
%
%
%

Finally, we prove Theorem \ref{th:intro2} (and its more precise version Theorem \ref{th:M-Lgen}) in Section \ref{sec:Mittag-Leffler} by a recursive application of Theorem \ref{th:MT}, combined with a standard procedure for ensuring the completeness of the limit immersion.


\section{Preliminaries and notation}\label{sec:prelim}

We write $\n=\{1,2,3,\ldots\}$, $\z_+=\n\cup\{0\}$, and $\imath=\sqrt{-1}$, and denote by $\cp^1=\c\cup\{\infty\}$ the Riemann sphere. We shall use the symbols $\Re$ and $\Im$ to denote, respectively, the real and the imaginary part, and identify $\c^n$ with $\r^{2n}$. 
We denote by $|\cdot|$ the Euclidean norm in $\r^n$.
Given maps $f,g\colon X\to Y$ between sets, we write $f\equiv g$ to mean that $f(x)=g(x)$ for all $x\in X$; we write $f\not\equiv g$ otherwise. The uniform norm (or $\sup$ norm) of  a map $f\colon X\to \r^n$ on $X$ is the non-negative number
\[
\|f\|_{X}=\sup  \{|f(x)|\colon x\in X\}.
\]
If $f,g\colon X\to \r^n$ are maps,   the notation $f\approx g$ shall mean that $\|f-g\|_X$ is so close to $0$   that no significant deviation  between $f$ and $g$  can be found  in a given argumentation. In this case we say that $f$ approximates $g$ on $X$.
\begin{definition}\label{de:full}
Given a set $X$, a map   $X\to \c^n$  is said to be {\em full} if its image lies in no affine hyperplane of $\C^n$. Full maps $X\to\cp^n$ are defined in the same way. 
\end{definition}
Assume that $X$ is a topological space.  A {\em Jordan arc} in  $X$ is an embedding $ [0,1]\to X$; an open Jordan arc in $X$ is an embedding  $(0,1)\to X$.  Continuous maps $C\colon \s^1\to X$ are said to be closed curves in $X$;  if in addition $C$ is an embedding then  the closed curve  is said to be simple or a {\em Jordan curve}. Usually we shall identify arcs and curves with their image.  We denote by $H_1(X,\z)$ the first homology group with integer coefficients on $X$.  

A smooth surface is said to be {\em open} if it is not compact and has no boundary. Throughout this paper surfaces are considered to have no boundary unless the contrary is indicated.
Assume that $M$ is an (open) smooth surface.
%
%
\begin{definition}\label{de:Runge}
A nonempty (possibly disconnected) compact set $S$ in  $M$ is said to be {\em Runge} if $M\setminus S$ has no relatively compact components.
\end{definition}
%
%
\begin{definition}\label{def:admissible}
A nonempty (possibly disconnected) compact set $S$ in  $M$ is called {\em admissible} if  
it is of the form $S=K\cup \Gamma$, where $K$ is a (possibly empty)
finite union of pairwise disjoint compact domains with smooth boundaries in $M$ 
and $\Gamma = \overline{S \setminus K}$ is a (possibly empty) union of finitely many pairwise disjoint Jordan curves in $S\setminus K$ and 
smooth Jordan arcs   in $\overline{S \setminus K}$ meeting $K$ only in their endpoints
(or not at all) and such that their intersections with the boundary 
$bK$ of $K$ are transverse.
\end{definition} 
%
%
\begin{definition}\label{de:skeleton}
Let $S=K\cup \Gamma$ be a connected admissible subset in   $M$ with  $K\neq \emptyset$, and fix a point $p_0\in \mathring K=\mathring S$ and a (possibly empty) finite subset $A \subset \mathring K\setminus \{p_0\}$ of cardinal $m\in \z_+$.  A family of smooth curves
\[
	\{C_j\colon j=1,\ldots,l\}, \quad l=m+\dim H_1(S,\z),
\]
 is said to be a  {\em skeleton of $S$ based at $(p_0,A)$} if the following conditions hold.
 \begin{enumerate}[\rm ({A}1)]
 \item $C_j\colon [0,1]\to S$ is a Jordan arc with   $C_j(0)=p_0$,  $C_j(1)\in A$, and $C_j([0,1))\cap A=\varnothing$,   $j=1,\ldots,m$, and  $A= \{ C_1(1),\ldots,C_m(1)\}$.
 \smallskip
 \item $C_j\colon \s^1\to  S$  is a closed curve  containing $p_0$ and disjoint from $A$, $j=m+1,\ldots,l$. These curves do not need to be simple.
 \smallskip
\item  $\{C_{m+1},\ldots,C_l\}$ 
determines  a basis of the homology group $H_1(S,\z)$.  
 \smallskip
 \item  $C=\bigcup_{j=1}^l C_j$ is a strong deformation retract of $S$.
  \smallskip 
\item  There is a  Jordan arc $\gamma_j\subset (C_j\cap \mathring K\setminus A)\setminus \big(\bigcup_{i\neq j } C_i\big)$ such that  $C_j|_{C_j^{-1}(\gamma_j)}$ is injective, $j=1,\ldots,l$.
\end{enumerate}
\end{definition}
By basic topology, every admissible subset $S\subset M$  in the assumptions of Definition \ref{de:skeleton} carries skeletons based at any  pair $(p_0,A)$ as above. Furthermore, if $\Gamma=\varnothing$ then the skeleton can be chosen such that $C\subset \mathring K$ and $C_i\cap C_j=\{p_0\}$ for all $i\neq j$. On the other hand, if $S$ is Runge in $M$ then $C$ is Runge in $M$ as well by {\rm (A4)}.

\subsection{Divisors and function spaces}\label{ss:spaces}
Given a  set $X$, we denote by $\Div(X)$ the free commutative group of {\em finite} divisors of $X$ with multiplicative notation:
\[
	\Div(X)=\Big\{\prod_{j=1}^k q_j^{n_j}\colon k\in \n,\; q_j\in X,\; n_j\in \z \Big\}.
\]
Here,  $q^0=1$ for all $q\in X$. 
Given $D=\prod_{j=1}^k q_j^{n_j}\in\Div(X)$,  the set $\supp  (D)=\{q_j\colon n_j\neq 0\}\subset X$ is said to be the {\em support} of $D$. The divisor $D$ is said to be {\em effective} if $n_j\geq 0$ for all $j=1,\ldots,k$. We write $D_1\geq D_2$ to mean that $D_1 D_2^{-1}$ is   effective.  

Let  $M$ and $N$ be a pair of complex manifolds and $S\subset M$ be a subset. We denote by $\Ccal^0(S,N)$   the space of continuous maps $S\to N$, and  write  $\Ccal^0(S)=\Ccal^0(S,\c)$. As it is customary, we denote by $\Ocal(S,N)$ the space of all holomorphic maps from some neighborhood of $S$ in $M$ (depending on the function) into $N$.  

We assume in the sequel that $M$ is a Riemann surface (either open or compact). For any subset $S\subset M$ we denote $\Ocal(S)=\Ocal(S,\c)$, whereas $\Ocal_\infty(S)$ will denote  the space of all meromorphic functions on some neighborhood of $S$ in $M$.   For a finite subset $E\subset \mathring S$, we denote 
\begin{equation}\label{eq:OcalinftySE}
	 \Ocal_\infty(S|E)=\Ocal_\infty(S)\cap \Ocal(S\setminus E);
\end{equation}
i.e., $\Ocal_\infty(S|E)$ is the space of all meromorphic functions on a neighborhood of $S$ which have poles (if any) only at points in $E$.
Likewise, we denote by $\Omega(S)$  the space of all holomorphic $1$-forms on some neighborhood of $S$ in $M$, $\Omega_\infty(S)$  the space of all meromorphic $1$-forms on some neighborhood of $S$ in $M$, and   $\Omega_\infty(S| E)=\Omega_\infty(S)\cap\Omega(S\setminus E)$.

Assume that the set $S\subset M$ is compact. For any $f\in \Ocal_\infty(S)$, $f\not\equiv 0$, $f\not\equiv\infty$, we call $Z(f)$ and $ P(f)$   the (finite) sets of zeros and poles of $f$ in $S$, and write $s_p\in \n$  for the zero or pole order of $f$ at $p$ for all $p\in Z(f)\cup P(f)$ . We set 
 \[
 	[f]_0=\prod_{p\in Z(f)} p^{s_p}\quad \text{and}\quad [f]_\infty=\prod_{p\in P(f)} p^{s_p}
\]
the effective divisors in $\Div(S)$ of zeros and poles of $f$ in $S$, respectively, and   
\[
	[f]:=\frac{[f]_0}{[f]_\infty}\in \Div(S)
\]  the divisor of $f$ in $S$. (We do not use the more customary parenthetical notation for divisors in order to avoid ambiguities.) We use the same notation for the corresponding divisors of a nonzero meromorphic $1$-form on $S$. 

For each effective divisor $D\in \Div(S)$   we set
\begin{equation}\label{eq:OcalDS}
	\Ocal_{D}(S)=\{f\in   \Ocal(S)\colon  [f]\geq D\}.
\end{equation}
 Assume that $S=K\cup \Gamma\subset M$ is  admissible in the sense of Definition \ref{def:admissible}. We denote 
 \[
 	\Acal(S,N)=\Ccal^0(S,N)\cap \Ocal(\mathring S,N) \quad\text{and}\quad \Acal(S)=\Acal(S,\c).
\]
For an effective divisor $D\in \Div(\mathring S)$, we denote 
\begin{equation}\label{eq:AcalD(S)}
\Acal_D(S)=\Acal (S)\cap \Ocal_D(U),
\end{equation}
 where $U\subset \mathring S$ is any compact neighborhood of $\supp (D)$. We also call  
 \[
 \Acal_\infty(S)=\Ccal^0(S,\cp^1)\cap \Ccal^0(bS,\c)\cap \Ocal_\infty(\mathring S),
 \] 
 and, given a finite set $E\subset\mathring S$, 
\[
	\Acal_\infty(S|E)=\Acal_\infty(S)\cap \Ocal(\mathring S\setminus E).
\]
Note that if $f\in\Acal_\infty(S)$ then, by the identity principle, $f$ has at most finitely many poles all which lie in $\mathring S$. 

Let $n\ge 3$ be an integer. We denote by
 \begin{equation}\label{eq:nullquadric1}
	\Sgot^{n-1}_*=\Big\{(u_1\ldots,u_n)\in \c^n\setminus\{0\}\colon u_1 u_2 = \sum_{j=3}^n u_j^2\Big\}.
\end{equation}
The punctured complex quadric $\Sgot^{n-1}_*$ is canonically identified with the punctured null quadric
 \begin{equation}\label{eq:nullquadric}
\Agot^{n-1}_*=\Big\{(z_1\ldots,z_n)\in \c^n\setminus\{0\}\colon \sum_{j=1}^n z_j^2=0\Big\}
\end{equation}
by the natural linear biholomorphism  $\Xi \colon \c^n\setminus\{0\}\to \c^n\setminus\{0\}$ given by
 \begin{equation}\label{eq:Thetabiho}
  \Xi(z_1,z_2,\ldots,z_n):=\big(-z_1+ \imath z_2,z_1+\imath z_2,z_3,\ldots, z_n\big),
   \end{equation}
which maps $\Agot^{n-1}_*$ into $\Sgot^{n-1}_*$. The punctured null quadric $\Agot_*^{n-1}$ is a complex homogeneous manifold, and hence an Oka manifold (see \cite[Example 5.6.2]{Forstneric2017}).

For any map $f=(f_1,\ldots,f_n)\in\Ccal^0(S,\cp^1)^n$ write
\begin{equation}\label{eq:finfi}
f^{-1}(\infty)=\bigcup_{j=1}^n f_j^{-1}(\infty),
\end{equation}
 and denote
 \[
\Ocal_\infty(S,\Sgot^{n-1}_*)=\{f\in \Ocal_\infty(S)^n\colon f\big(S\setminus f^{-1}(\infty)\big)\subset \Sgot^{n-1}_*\}
\]
and 
\[
\Acal_\infty(S,\Sgot^{n-1}_*)=\{f\in \Acal_\infty(S)^n\colon f\big(S\setminus f^{-1}(\infty)\big)\subset \Sgot^{n-1}_*\}.
\]
 Note that $f^{-1}(\infty)\subset\mathring S$ for all $f\in \Acal_\infty(S,\Sgot^{n-1}_*)$.   Given a finite set $E\subset\mathring S$,  we denote
  \[
  \Ocal_\infty(S|E,\Sgot^{n-1}_*)=  \Ocal_\infty(S,\Sgot^{n-1}_*)\cap   \Ocal(S\setminus E,\Sgot^{n-1}_*)
  \] 
  and 
    \[
  \Acal_\infty(S|E,\Sgot^{n-1}_*)=  \Acal_\infty(S,\Sgot^{n-1}_*)\cap   \Ocal(\mathring S\setminus E,\Sgot^{n-1}_*).
  \]
  
  The spaces $\Ocal_\infty(S,\Agot^{n-1}_*)$, $\Acal_\infty(S,\Agot^{n-1}_*)$, $\Ocal_\infty(S|E,\Agot^{n-1}_*)$, and $\Acal_\infty(S|E,\Agot^{n-1}_*)$ are defined in the same way. 


\subsection{Minimal surfaces in $\r^n$}\label{sec:minimal}
Let $M$ be an open Riemann surface. 
 A map  $X=(X_1,\ldots,X_n)\colon M\to\r^n$ $(n\ge 3)$  is   a  conformal minimal immersion if and only if  $X$ is harmonic, and its complex derivative $\di X\in\Omega(M)^n$ (i.e., the $(1,0)$-part of the exterior differential $dX$ of $X$) vanishes nowhere on $M$ and satisfies $\sum_{j=1}^n\di X_j^2\equiv 0$. Given a holomorphic $1$-form $\theta$ on $M$ with no zeros, the last two conditions are equivalent to $\di X/\theta\in \Ocal(M,\Agot^{n-1}_*)$ or to $\Xi\big( \di X/\theta\big)\in \Ocal(M,\Sgot^{n-1}_*)$; see \eqref{eq:Thetabiho}. Moreover, in this case $X\colon M\to\r^n$ is given by the formula
\[
	M\ni p\longmapsto X(p_0)+\Re\int_{p_0}^p 2\di X
\]
for any fixed base point $p_0\in M$. This is known in the literature as the {\em Enneper-Weierstrass representation} formula for minimal surfaces in $\r^n$ (see e.g.\ \cite{Osserman1986}).
 A map $X:S\to \r^n$ from a subset $S\subset M$ is said to be a {\em conformal minimal immersion} if it extends as a conformal minimal immersion to some open neighborhood of $S$ in $M$ (depending on $X$).

%
%
%
\begin{definition}\label{def:full-CMI}
We say that a conformal minimal immersion $X:M\to\r^n$ is {\em full} if 
the holomorphic map $\di X/\theta\colon M\to\c^n$ is full for any holomorphic $1$-form $\theta$ vanishing nowhere on $M$.
\end{definition}
The group homomorphism $\Flux_X\colon H_1(M,\z)\to\r^n$ given by
\[
	\Flux_X(\gamma)=2\int_\gamma\Im (\di X)=-2\imath\int_\gamma\di X
	\quad
	\text{for every loop $\gamma\subset M$},
\]
is said to be the {\em flux map} (or just the {\em flux}) of $X$. 
A conformal minimal immersion $X:M\to \r^n$  is said to be of {\em finite total curvature} (acrostically, FTC) if
\[
\int_M K\, dA>-\infty,
\]
where $K$ is the Gauss curvature of the Riemannian metric 
\begin{equation}\label{eq:ds2}
	ds^2=2\sum_{j=1}^n |\di X_j|^2
\end{equation}
induced on $M$ by the Euclidean one via $X$, and $dA$ is the area element of $ds^2$. 

The conformal minimal immersion $X:M\to\r^n$ is called {\em complete} if the metric $ds^2$ is complete in the classical sense of Riemannian geometry. Assume now that $M$ is a Riemann surface with compact boundary $bM\subset M$ (possibly $bM=\varnothing$). If $M$ carries a {\em complete} conformal minimal immersion $X:M\to \r^n$ of FTC (here, $X$ is complete if and only if  $X\circ \gamma$ has infinite Euclidean length for any divergent arc $\gamma\colon[0,1)\to M$), then the classical results by Huber \cite{Huber1957CMH} and Chern-Osserman \cite{ChernOsserman1967JAM} impose the following conditions.
\begin{enumerate}[\rm (i)]
\item $M$ is conformally equivalent to $R\setminus E$, where $R$ is  a compact Riemann surface with compact boundary $bR$ and $E\subset R\setminus bR$ is a finite subset.
\smallskip
\item $\di X$ extends meromorphically to $R$ with an effective pole at each point of $E$. 
\end{enumerate}
Conversely, if $R$ and $E$ are as in {\rm (i)} and if $X\colon R\setminus E \to \r^n$ is a conformal minimal immersion satisfying {\rm (ii)}, then $X$ is complete and of FTC.
This discussion justifies the following definition.
\begin{definition} \label{def:GCMI}  
Let $S=K\cup \Gamma$ be an admissible set in an open Riemann surface $M$ (see Definition \ref{def:admissible}), let $E\subset \mathring K=\mathring S$ be a finite subset,
and let $\theta$ be a nowhere vanishing holomorphic $1$-form on $M$. 
A {\em generalized complete conformal minimal immersion of finite total curvature}  
from $S\setminus E$ into $\R^n$ $(n\ge 3)$ is a pair $(X,f\theta)$, 
where $X\colon S\setminus E\to \r^n$ is a $\Ccal^1$ map that is a conformal 
minimal immersion on $\mathring K\setminus E$ and $f$ is a map in $\Acal_\infty(S|E,\Agot^{n-1}_*)$ satisfying the following conditions.
\begin{enumerate}[\rm (i)]
\item $f\theta =2\di X$ holds on $K\setminus E$.
\smallskip
\item For any smooth path $\alpha$ in $M$ parameterizing a connected component of 
$\Gamma$, we have $\Re(\alpha^*(f\theta))=\alpha^*(dX)=d(X\circ \alpha)$.
\smallskip
\item  $f^{-1}(\infty)=E$; see \eqref{eq:finfi}.
\end{enumerate}

We denote by 
\begin{equation}\label{eq:GCCMI}
	\GCMI_\infty(S|E,\R^n)
\end{equation}
the space of all generalized complete conformal minimal immersions of FTC from $S\setminus E$ into $\R^n$, and write  $2\partial \hat X=f\theta$ for all $\hat X=(X,f\theta)$ in $\GCMI_\infty(S|E,\R^n)$. If $\hat X=(X,f\theta)$, $\hat Y=(Y,g\theta)\in \GCMI_\infty(S|E,\r^n)$,  the notation $\hat X\approx \hat Y$ on $S$ means that $X-Y\approx 0$ and $f-g\approx 0$ on $S\setminus E$, and hence are continuous on $S$.
For $\hat X=(X,f\theta)\in \GCMI_\infty(S|E,\R^n)$, the {\em flux map}  $\Flux_{\hat X}\colon H_1(S\setminus E,\z)\to\r^n$ of $\hat X$  is the group homomorphism given by
\[
\Flux_{\hat X}(\gamma):= \int_\gamma\Im (f \theta)=- \imath \int_\gamma f\theta, \quad
	\text{for every loop $\gamma\subset S$}.
\]
\end{definition}
\begin{remark}
Given  $\hat X=(X,f\theta)\in \GCMI_\infty(S|E,\R^n)$, the (well defined) map 
$[f_1:\cdots:f_n]:S\setminus E\to \c\P^{n-1}$   extends holomorphically  to the punctures $E\subset \mathring S$, and hence it lies in $\Acal(S,\cp^{n-1})$.
\end{remark}
Finally, as above, we denote by 
\begin{equation}\label{eq:CCMI}
	\CMI_\infty(S|E,\R^n)
\end{equation}
the subspace of those immersions $(X,f\theta)\in\GCMI_\infty(S|E,\r^n)$ such that $X$ extends as a conformal minimal immersion to some neighborhood of $S\setminus E$ in $M$; in this case, we just write $X$ for $(X,f\theta=2\di X)$.


\section{Mergelyan's theorem for complete minimal surfaces of finite total curvature}\label{sec:Mergelyan}
In this section we  prove a  Mergelyan type theorem for complete minimal surfaces with FTC, asserting that generalized complete conformal minimal immersions of finite total curvature on a finitely punctured admissible subset, $S\setminus E_0$, can be approximated uniformly on $S$ by complete conformal minimal immersions on a neighborhood of $S\setminus E_0$.

\begin{theorem}\label{th:Merg} 
Let $M$ be an open  Riemann surface, $\theta$ be a nowhere vanishing holomorphic $1$-form on $M$,   and $S=K\cup\Gamma\subset M$ be a connected   admissible subset (see Definition \ref{def:admissible}). Also let $E_0$ and $\Lambda$ be a pair of disjoint finite subsets of $\mathring S$ and let $n\geq 3$ be an integer.
For any  $\hat X=(X,f\theta)\in \GCMI_\infty(S|E_0,\R^n)$, any number $\epsilon>0$, and any integer $r\geq 0$, there is $Y\in \CMI_\infty(S|E_0,\R^n)$ satisfying the following conditions.
\begin{enumerate}[{\rm (i)}]
\item $Y$ is full (see Definition \ref{def:full-CMI}).
\smallskip
 \item $Y-X$ extends to $S$ as a continuous map and $\|Y-X\|_{S}<\epsilon$.
\smallskip
\item $Y-X$ vanishes at least to order $r$ at every point of $\Lambda\cup E_0$.
\smallskip
\item $\Flux_Y=\Flux_{\hat X}$.
\end{enumerate}
\end{theorem}
The improvement of $Y$ with respect to $X$ is that $Y$ is a true conformal minimal immersion in a neighborhood of $S$, and not just a generalized one on $S$. Moreover, we ensure that $Y$ is full.
%
%
%
\begin{proof} 
We adapt the arguments in \cite{AlarconForstneric2014IM} (see also \cite{AlarconForstnericLopez2016MZ}) to the special framework of complete minimal surfaces of FTC. We start with the following reduction.
\begin{claim}\label{cl:posgen} 
We can assume that $K\neq \varnothing$, $X|_U$ is flat on no component $U$ of $K$, and there is a component $U_0$ of $K$ such that $X|_{U_0}$ is full.
\end{claim}
Recall that $X|_U$ is flat if and only if $X(U)$ lies in an affine plane in $\r^n$, or equivalently, $(\di X/\theta)(U)\subset \Agot^{n-1}_*$ lies in a complex line in $\c^n$.
\begin{proof}
We first show that we can assume that $K\neq\varnothing$ and $f$ (and hence, $X$) is full on a component $U_0$ of $K$. Indeed, if $K=\varnothing$, then $S=\Gamma$ is either a Jordan arc or a Jordan curve; recall that $S$ is connected. Choose a closed disc $U_0$ in $M$ so small that $U_0\cap S$ is a Jordan arc and  $\hat X$ is approximately constant there. Up to a slight deformation of $\hat X$ on a small neighborhood of $U_0\cap S$ (see \cite[Lemma 3.3]{AlarconCastro-Infantes2019APDE} for details on how to make the deformation), we can extend it, with the same name and flux, to $S\cup U_0$ as a generalized conformal minimal immersion such that
\begin{equation}\label{eq:ffull*} 
f|_{U_0}\quad \text{is full}.
\end{equation} 
Indeed it suffices to replace $\hat X|_{U_0}$ by a full conformal minimal immersion on $U_0$ that is close to $\hat X|_{U_0\cap S}$ and suitably perturb $\hat X$, by using  \cite[Lemma 3.3]{AlarconCastro-Infantes2019APDE}, on the complement of $U_0$ in a small neighborhood of $U_0\cap S$ in $S$.

Suppose now that $U\subset K$ is a component and $X|_U$ is flat (obviously $U\neq U_0$). This means that  $f(U)$ is   contained in a complex line $L=\c v$, where $v\in \Agot^{n-1}_*$. Up to a rigid motion in $\r^n$ we can suppose that 
$v=(1,\imath ,0\ldots,0)$, and hence $f=(f_1,\imath f_1,0\ldots,0)$,
where $f_1\in \Acal_\infty(U|E_0)$ vanishes nowhere on $U$ and $f_1 \theta$ is exact on $U\setminus E_0$ (see Subsec.\ \ref{ss:spaces}).
Fix $r_0\geq r+\sum_{p\in E_0\cap U} {\rm Ord}_p(f_1)$ and call  $\Delta_U=\prod_{p\in (E_0\cup \Lambda)\cap U} p^{r_0}$, where ${\rm Ord}_p(f_1)> 0$ denotes the pole order of $f_1$ at $p\in E_0\cap U$. Fix a point $q\in U\setminus E_0$. Since the complex space   $\Ocal_{\Delta_U}(U)$ (see \eqref{eq:OcalDS}) has  infinite  dimension, there is $h\in \Ocal_{\Delta_U}(U)$, $h\not\equiv 0$, satisfying the following conditions.
\begin{enumerate}[\rm (a)]
\item The $1$-forms $h^2 f_1 \theta$ and $hf_1\theta$ are exact on $U$.
\smallskip
\item $\int_{q}^p (hf_1,h^2f_1)\theta=0$ for all $p\in U\cap(E_0\cup \Lambda)$.
\smallskip
\item The functions $1, h,h^2$ are $\c$-linearly independent. 
\end{enumerate}
Set
\[
f_\zeta:=(f_1(1-\zeta^2 h^2),\imath f_1(1+\zeta^2 h^2),2 \zeta hf_1,0,\ldots,0),\quad \zeta\in \c,
\]
and define
\[
X_\zeta(p)=X(q)+\Re\int_q^p f_\zeta \theta,\quad p\in U\setminus E_0.
\]
Note that, by {\rm (a)}, $f_\zeta\theta$ is exact, and hence $X_\zeta$ is well defined and  $\Flux_{X_\zeta}=\Flux_{X|_U}=0$. Moreover, if $\zeta\neq 0$ is chosen close enough to $0\in\c$, then   $f_\zeta\in \Acal_\infty(U| E_0,\Agot^{n-1}_*)$,  $X_\zeta\in \GCMI_\infty(U|E_0,\r^n)$, $\|f_\zeta-f\|_U\approx 0$, and
  $\|X_\zeta-X\|_U\approx 0$ (see Sect.\ \ref{sec:prelim} for notations). Furthermore, $X_\zeta-X$ vanishes at least to order $r$ at every point of $(\Lambda\cup E_0)\cap U$ (see {\rm (b)} and recall that $h\in\Ocal_{\Delta_U}(U)$), and $f_\zeta(U)$ is not contained in a complex line (see {\rm (c)}). Up to replacing $X|_U$ by $X_\zeta$ in $\hat X$, and then slightly modifying $f|_\Gamma$ on a small neighborhood of $U$ preserving the smoothness and the flux map of $\hat X$ (use \cite[Lemma 3.3]{AlarconCastro-Infantes2019APDE}), we can suppose that $f(U)$ is not contained in a complex line. To finish the proof, we apply the same procedure in each component of $K$ on which $f$ is flat.
\end{proof}

Assume, as we may, that the hypotheses of Claim \ref{cl:posgen} hold. 

Denote by $m\in\z_+$ the cardinal of $E_0\cup\Lambda$. Fix a point $p_0\in \mathring K\neq\varnothing$ and let  $\{C_1,\ldots,C_l\}$, $l=m+\dim H_1(S,\z)$,  be a skeleton of $S$ based at $(p_0, E_0\cup \Lambda)$; see Definition \ref{de:skeleton}.  Write  $C=\bigcup_{j=1}^lC_j$ and denote  
\[
	\Ccal^0(C,f)=\{h\in \Ccal^0(C\setminus E_0,\Agot^{n-1}_*)\colon h-f\in\Ccal^0(C,\c^n)\}
\]
(see Subsec. \ref{ss:spaces}) and let $\Qcal=(\Qcal_1,\ldots,\Qcal_l) \colon \Ccal^0(C,f)\to (\c^n)^l$ be the period map defined by
 \begin{equation}\label{eq:Qcal-map*}
	  \Ccal^0(C,f)\ni h  \longmapsto\Qcal(h)=\Big(  \int_{C_j} (h-f)\theta \Big)_{j=1,\ldots,l}.
\end{equation} 
Fix a function $g\in \Ocal(M)$ with $[g]=\prod_{p\in E_0}p^{o(p)}$, where $o(p)=\max\{{\rm Ord}_p(f_j)\colon j=1,\ldots,n\}$; such a function exists by the classical Weierstrass theorem (on the existence of holomorphic functions with prescribed divisor on an open Riemann surface); see \cite{Florack1948}. Since $f\in \Acal_\infty(S|E_0,\Agot_*^{n-1})$, we have that 
\begin{equation}\label{eq:f0gfM}
	f_0:=g f\in \Acal(S,\Agot^{n-1}_*).
\end{equation}
Consider a family of complete holomorphic vector fields $V_1,\ldots,V_m$ on $\c^n$, vanishing at $0$, tangential to $\Agot^{n-1}_*$ along $\Agot_*^{n-1}$, and such that $\{V_1(z),\ldots,V_m(z)\}$ spans the tangent space $T_z\Agot_*^{n-1}$ for all $z\in \Agot^{n-1}_*$ (see e.g.\ \cite[Example 4.4]{AlarconForstneric2014IM}). Obviously, $m\ge n$. For each $j=1,\ldots, l$, let $\gamma_j\subset C_j\cap \mathring K\setminus (E_0\cup\Lambda\cup \{p_0\})$ be a Jordan arc satisfying  {\rm (A5)}  in Definition \ref{de:skeleton}. Since $\gamma_j$ lies in a  component of $K$ and $f$ is assumed to be flat on no component of $K$, there are pairwise distinct points  $p_{1,j},\ldots,p_{m,j}\in \gamma_j$  such that  
\begin{equation}\label{eq:baseV*}
\text{$\{V_1(f_0(p_{1,j})),\ldots,V_m(f_0(p_{m,j}))\}$ spans $\c^n$};
\end{equation}
see \eqref{eq:f0gfM} and take into account the geometry of $\Agot_*^{n-1}$. 
Also choose functions $f_{1,j},\ldots,f_{m,j}\in \Ccal^0(C,\c)$, with pairwise disjoint supports, such that $p_{i,j}$ lies in the relative interior of $\supp (f_{i,j})\subset \gamma_j$ and
\begin{equation}\label{eq:perV*}
	\int_{C_j} f_{i,j} (V_i\circ f_0) \theta \approx V_i(f_0(p_{i,j}))\quad \text{for all } i=1,\ldots,m.
\end{equation}
Set $F=(f_{1,j}\ldots,f_{m,j})_{j=1,\ldots,l}\in \Ccal^0(C,(\c^m)^l)$. Denote by $\phi_t^i$   the flow of $V_i$ over $\Agot^{n-1}_*$, $i=1,\ldots,m$. Let $\Phi_F\colon (\c^m)^l\times S\times \Agot^{n-1}_*\to \Agot^{n-1}_*$ be defined by 
\[
\Phi_F(\zeta, p, z):=(\phi^1_{\zeta_{1,1} f_{1,1}(p)} \circ \cdots \circ \phi^m_{\zeta_{m,1} f_{m,1}(p)}\circ\cdots\circ \phi^1_{\zeta_{1,l} f_{1,l}(p)} \circ \cdots \circ \phi^m_{\zeta_{m,l} f_{m,l}(p)})(z),
\]
where  $\zeta=\big((\zeta_{i,j})_{i=1,\ldots,m}\big)_{j=1,\ldots,l}$, and the spray with core $f_0$ given by
\[
\Phi_{F,f_0}\colon (\c^m)^l \times S\to \Agot^{n-1}_*,\quad \Phi_{F,f_0}(\zeta,p):=\Phi_F(\zeta, p, f_0(p)).
\]
By the choice of $f_{i,j}$, we have that $\Phi_{F,f_0}(\zeta,\cdot)/g$ is continuous on $C\setminus E_0$ and coincides with $f$ on a neighborhood of $E_0\cup\Lambda$ in $C$, namely, in $C\setminus \bigcup_{i,j}\supp (f_{i,j})$, and hence $\Phi_{F,f_0}(\zeta,\cdot)/g\in \Ccal^0(C,f)$, for all $\zeta \in (\c^m)^l$. Consider the new period map $\Qcal^*=(\Qcal^*_1,\ldots,\Qcal^*_l)\colon (\c^m)^l\to (\c^n)^l$ given by
\[
	\Qcal^*(\zeta)=\Qcal\Big(\frac{\Phi_{F,f_0}(\zeta,\cdot)}{g}\Big),\quad \zeta\in(\c^m)^l.
\]
The spray $ \Phi_{F,f_0}$ is {\em $\Qcal^*$-dominating} at $\zeta=0$ in the sense that the Jacobian matrix
\[
W:=\Big(\big(\frac{\partial \Qcal^*}{\partial \zeta_{i,j}}\Big|_{\zeta=0}\big)_{i=1,\ldots,m}\Big)_{j=1,\ldots,l}
\]
has maximal rank equal to $nl$. Indeed,  since
\[
\Big(\frac{\partial \Qcal^*_k}{\partial \zeta_{i,j}}\Big|_{\zeta=0}\Big)_{i=1,\ldots,m}=(0)_{m\times n}\quad \text{if $k\neq j$}
\]
and, by \eqref{eq:perV*},
\[
W_j:=\big(\frac{\partial \Qcal^*_j}{\partial \zeta_{i,j}}\Big|_{\zeta=0}\big)_{i=1,\ldots,m}\approx (V_i(f_0(p_{i,j}))_{i=1,\ldots,m}\quad \text{for all $j=1,\ldots,l$},
\] 
the block structure of $W$,  \eqref{eq:baseV*}, and \eqref{eq:perV*}  guarantee that  ${\rm rank} (W)=\sum_{j=1}^l{\rm rank} (W_j) =nl$, provided that the approximation in \eqref{eq:perV*} is sufficiently close. So, there is a small closed ball $V$ around the origin of $(\c^m)^l$ such that the holomorphic map
\begin{equation}\label{eq:Q*}
\text{$\Qcal^*:V\to \Qcal^*(V)$ is a submersion with $\Qcal^*(0)=\Qcal(f)=0$};
\end{equation}
take into account \eqref{eq:f0gfM}. Moreover, we choose $V$, as we may by continuity, so small that
\begin{equation}\label{eq:Cauchy-estimates}
\Phi_{F,f_0}(\zeta,\cdot)\approx f_0\quad \text{for all }\zeta\in V;
\end{equation} recall that $f_0$ is the core of $\Phi_{F,f_0}$.
 Write $f=(f_1,\ldots,f_n)$ and fix $r_0\in \n$ with
\begin{equation} \label{eq:r0*}
r_0\geq r+\sum_{p\in E_0}\sum_{i=1}^n{\rm Ord}_p(f_i),
\end{equation}
where $r\ge 0$ is the integer given in the statement of the theorem and ${\rm Ord}_p(\cdot) $ means pole order at $p\in E_0$; recall that  $f_i\in \Acal_\infty(S|E_0)$ for all $i=1,\ldots,n$. Next, since $C$ is a Runge subset of any neighborhood of $S$ (see {\rm (A4)} in Definition \ref{de:skeleton}), the classical Runge-Mergelyan theorem with jet-interpolation enables us to approximate each $f_{i,j}$ uniformly on $C$ by a function $h_{i,j}\in\Ocal(S)$ satisfying
\begin{equation}\label{eq:hjifin*}
[h_{i,j}]\geq \Delta_1:=\prod_{p\in E_0\cup \Lambda} p^{r_0+1};
\end{equation}
recall that $f_{i,j}\equiv 0$ on a neighborhood of $E_0\cup\Lambda$. 
Consider the map $H=(h_{1,j},\ldots,h_{m,j})_{j=1,\ldots,l}\in \big(\Ocal_{\Delta_1}(S)^m\big)^l$; see \eqref{eq:OcalDS}. Likewise, since $\Agot^{n-1}_*$ is an Oka manifold (see \cite[Example 4.4]{AlarconForstneric2014IM}; see Forstneri\v c \cite{Forstneric2017} for a comprehensive monograph in the subject), by \eqref{eq:f0gfM} there is  $h_0\in \Ocal(S,\Agot^{n-1}_*)$ such that
\begin{equation}\label{eq:h0f0fin*}
\text{ $h_0\approx f_0$ on $S$ and $h_0-f_0$ vanishes to order $r_0+1$ on $ E_0\cup \Lambda$.}
\end{equation}
Therefore, $\Phi_{H,h_0}\approx\Phi_{F,f_0}$ uniformly on $V\times S$, where  $\Phi_{H,h_0}\in \Ocal(V\times S, \Agot^{n-1}_*)$ is  the holomorphic spray with core $h_0$  given by
 \begin{multline*}
\Phi_{H,h_0}(\zeta, p):=
\\
\big(\phi^1_{\zeta_{1,1} h_{1,1}(p)} \circ \cdots \circ \phi^m_{\zeta_{m,1} h_{m,1}(p)}\circ\cdots\circ \phi^1_{\zeta_{1,l} h_{1,l}(p)} \circ \cdots \circ \phi^m_{\zeta_{m,l} h_{m,l}(p)}\big)(h_0(p)).
\end{multline*}
We emphasize that $\Phi_{H,h_0}(\zeta,\cdot)$ is a holomorphic function on a neighborhood of $S$; this is the key achievement for the proof. On the other hand, \eqref{eq:f0gfM}, \eqref{eq:r0*}, \eqref{eq:hjifin*}, \eqref{eq:h0f0fin*}, and \cite[Lemma 2.2]{AlarconCastro-Infantes2019APDE} ensure that
\begin{equation}\label{eq:interpofin*}
	\frac{\Phi_{H,h_0}(\zeta,\cdot)}{g}-f\in \Acal_{\Delta_2} (S)^n\quad \text{for each  $\zeta \in V$},
\end{equation}
 where $\Delta_2:=\prod_{p\in E_0\cup \Lambda} p^r$; see \eqref{eq:AcalD(S)}. In particular, $\frac{\Phi_{H,h_0}(\zeta,\cdot)}{g}\in  \Ccal^0(C,f)$ and the   period map 
\[
\hat \Qcal\colon V\to (\c^m)^l,\quad \zeta\mapsto \hat \Qcal(\zeta)= \Qcal\big(\frac{\Phi_{H,h_0}(\zeta,\cdot)}{g}\big),
\]
is well defined. If all  the  approximations are chosen close enough, then, by the Cauchy estimates, $\hat \Qcal\approx \Qcal^*$ on $V$, and there is $\zeta_0\in \mathring V$ close to the origin such that $\hat\Qcal(\zeta_0)=0$ (see \eqref{eq:Q*}), and the map 
\[
\hat f:=\frac{\Phi_{H,h_0}(\zeta_0,\cdot)}{g}\in \Ocal_\infty(S|E_0,\Agot^{n-1}_*) 
\]
satisfies the following conditions.
\begin{enumerate}[\rm (A)]
\item $\hat f-f  \in \Acal_{\Delta_2} (S)^n$; see \eqref{eq:interpofin*}.
\smallskip
\item $\hat f-f$ is exact on $S$; use that  $\Qcal (\hat f)=\hat \Qcal (\zeta_0)=0$ and \eqref{eq:Qcal-map*}.
\smallskip
\item $\hat f-f\approx 0$ on $S$; use \eqref{eq:Cauchy-estimates}.
\smallskip
\item $\hat f:S\setminus E_0\to \Agot^{n-1}_*\subset \c^n$ is full; use \eqref{eq:ffull*} and {\rm (C)}.
\end{enumerate}
It follows that the conformal minimal immersion $Y\colon S\setminus E_0\to \r^n$ given by
\[
Y(p)=X(p_0)+\Re \int_{p_0}^p \hat f \theta,\quad p\in S\setminus E_0,
\]
satisfies the conclusion of the theorem. Indeed, conditions {\rm (i)}--{\rm (iv)} follow easily from {\rm (A)}--{\rm (D)}. On the other hand, the condition $Y\in\CMI_\infty(S|E_0,\r^n)$, and in particular $Y$ is complete, is implied by {\rm (A)} and the fact that $\hat X\in\GCMI_\infty(S|E_0,\r^n)$.
\end{proof}


\section{An extension of the Behnke-Stein-Royden theorem}\label{sec:Royden}

In this section we prove a Behnke-Stein-Royden type theorem (see Theorem \ref{th:BSR}) with extra control on the divisors of the approximating functions; see conditions {\rm (ii)} and {\rm (iii)} in the following result. Recall the notation introduced in Subection \ref{ss:spaces}.

%
%
\begin{proposition}\label{pro:Royden0}
Let $\Sigma$ be a compact Riemann surface, let  $E\subset \Sigma$ be a nonempty finite subset, and let $S\subset \Sigma\setminus E$ be a Runge admissible subset.
For  any $f\in \Acal_\infty(S)$, any effective divisor $D_1\in \Div(\mathring S)$,  any real number $\delta>0$, and any $k\in \n$, there exists $\wt f\in \Ocal_\infty(\Sigma)\cap \Ocal \big( \Sigma\setminus( S\cup E)\big)$ satisfying the following conditions.
 \begin{enumerate}[\rm (i)]
\item $\wt f-f\in \Acal_{D_1}(S)$ and $\|\wt f-f\|_{S}<\delta$. 
\smallskip
\item $[\wt f|_{\Sigma\setminus (S\cup E)}]=D_0^2$ for some divisor $D_0\in \Div(\Sigma\setminus (S\cup E ))$.
\smallskip
\item $[\wt f]_\infty\geq \prod_{p\in E} p^k$.
\end{enumerate}
In particular, $\wt f\in \Ocal_\infty(\Sigma|E\cup \supp([f]_\infty))$.

\end{proposition}
\begin{proof}    
We may assume without loss of generality  that $f\in \Ocal_\infty(S)\cap \Ocal(\overline S\setminus \mathring S)$. Indeed, since $\Sigma\setminus E$ is an open Riemann surface, the classical Weierstrass theorem \cite{Florack1948} gives a function $\varphi
\in  \Ocal(\Sigma \setminus E)$ with
$[\varphi]=[f|_{\mathring S}]_\infty$; recall that $f^{-1}(\infty)$ lies in $\mathring S$ and consists of finitely many points. It turns out that $\varphi f\in \Acal(S)$ and vanishes nowhere on $\supp \big([f|_{\mathring S}]_\infty\big)$.     Mergelyan theorem with jet interpolation   then provides  $\varphi_0\in \Ocal(S)$ approximating $\varphi f $ on $S$ and satisfying $[\varphi_0-\varphi f]\geq D_1 [f|_{\mathring S}]_\infty$. If the proposition is valid for $\varphi_0/\varphi\in   \Ocal_\infty(S)\cap  \Ocal(\overline S\setminus \mathring S)$, then the solution provided for this function solves the proposition for $f$ whenever that the approximation $\varphi_0\approx \varphi f$ is close enough.
 
So, assume that $f\in \Ocal_\infty(S)\cap \Ocal(\overline S\setminus \mathring S)$.   Since $f\in\Ocal_\infty(S)$, Theorem \ref{th:BSR} gives us a function $f_0\in \Ocal_\infty(\Sigma)\cap \Ocal(\Sigma\setminus (S\cup E))$ satisfying the following conditions.
\begin{enumerate}[\rm ({P}1)]
\item $f_0-f\in\Ocal_{D_1}(S)$.
\smallskip
\item $\|f_0-f\|_{S}<\delta$.
\end{enumerate} 
It turns out that $f_0$ satisfies condition {\rm (i)}; however, it does not need to satisfy {\rm (ii)} or {\rm (iii)}. The next step in the proof is to find $h_0\in \Ocal_\infty(\Sigma| E)$ such that $h_0f_0$ satisfies {\rm (i)} and {\rm (ii)}.

If $f_0$ vanishes nowhere on $\Sigma\setminus(S\cup E)$, then it suffices to choose $h_0\equiv1$ (and $D_0=1$). Otherwise, write $[f_0|_{\Sigma\setminus(S\cup E)}] =\prod_{j=1}^s p_j^{m_j}$, where $p_1,\ldots,p_s$ are pairwise distinct points and $m_j>0$ for all $j\in\{1,\ldots,s\}$ $(s\ge 1)$, and set
\begin{equation}\label{eq:defE1}
	E_1=\{p_j\colon m_j\text{ is odd}\}\subset \Sigma\setminus(S\cup E).
\end{equation}
If $E_1=\varnothing$, then, again, it suffices to choose $h_0\equiv 1$ (and $D_0=\prod_{j=1}^s p_j^{m_j/2}$). Assume that $E_1\neq\varnothing$.
Since $\Sigma\setminus E$ is an open Riemann surface, there is $g\in  \Ocal(\Sigma \setminus E)$ with
\[
	[g]=\prod_{p\in E_1} p;
\]
see  \cite{Florack1948}.
Moreover, by a standard application of Runge's theorem for holomorphic functions into $\C\setminus\{0\}$ (an Oka manifold), we can assume in addition that
\begin{equation}\label{eq:gapprox1}
	\|g-1\|_{S}<\frac12\quad \text{ on }S.
\end{equation}
Consider the open Riemann surface
\[
	R=\{(p,u)\in (\Sigma\setminus E)\times \c\colon u^2= g(p)\}
\]
and notice that $R$ admits a canonical analytical compactification, namely, $\hat R$. By analytical continuation arguments, the  projection $\pi\colon \hat R\to \Sigma$, $\pi(p,u)=p$, is a $2$-sheeted branched covering, $R= \pi^{-1}(\Sigma\setminus E)=\hat R\setminus \pi^{-1}(E)$,  and 
\begin{equation}\label{eq:ramipi}
\text{$\pi^{-1}(E_1)$ is   the ramification set of $\pi|_R$.}
\end{equation}
Denote by $A\colon \hat R\to \hat R$ the    deck transformation of $\pi$ and observe that $A(p,u)=(p,-u)$ for all $(p,u)\in R$, hence  $\pi^{-1}(E_1)$ is   the fixed point set of $A|_R\colon R\to R$ as well. In view of \eqref{eq:gapprox1}, it turns out that
  $\pi^{-1}(S)=S^+\cup S^-$ where $S^+$ and $S^-$ are pairwise disjoint Runge compact subsets of $R$, $A(S^+)=S^-$, and $\pi|_{S^\pm}\colon S^\pm\to S$ is a biholomorphism.
 
Given $\delta'>0$, Theorem \ref{th:BSR} furnishes us with a function $h\in \Ocal_\infty(\hat R)\cap \Ocal(R)$ satisfying the following conditions.
\begin{itemize}
\item  $\|h- (\pm1)\|_{S^\pm}<\delta'$.
\smallskip
\item  $h$ has simple zeros at all points in $\pi^{-1}(E_1)$.
\smallskip
\item $[h|_{S^\pm}-(\pm 1)]\geq D_1^\pm [(f_0\circ \pi)|_{S^\pm}]_\infty$  where $D_1^\pm\in \Div(S^\pm)$ is the only divisor with $\pi(D_1^\pm)=D_1$.
\end{itemize}
Up to replacing $h$ by $(h-h\circ A)/2$  we can also assume   that $h\circ A=-h$, and hence 
\begin{equation}\label{eq:hh0}
h^2=h_0\circ \pi\quad \text{for some $h_0\in \Ocal_\infty(\Sigma| E)$.}
\end{equation}
The following conditions are satisfied.
\begin{enumerate}[\rm ({P}1)]
\item[\rm ({P}3)]   $\|h_0-1\|_{S}<(\delta')^2+2\delta'$.
\smallskip
\item[\rm ({P}4)]  $[h_0|_{\Sigma\setminus E}]_0=D^2\prod_{p\in E_1}p$ for some divisor $D\in \Div(\Sigma\setminus (S\cup E))$; see  \eqref{eq:ramipi} and \eqref{eq:hh0}.
\smallskip
\item[\rm ({P}5)] $[h_0- 1]_0\geq D_1[f_0|_S]_\infty$.
\end{enumerate}
Set $f_1=h_0f_0$. Since $f_1-f=(h_0-1)f_0+(f_0-f)$, properties {\rm (P1)} and {\rm (P5)} ensure that $f_1-f\in\Ocal_{D_1}(S)$.  Since $S$ is compact, this, {\rm (P2)}, and {\rm (P3)} guarantee that 
\begin{equation}\label{eq:f1-f}
	\|f_1-f\|_{S}<\delta
\end{equation}
provided that $\delta'>0$ is chosen sufficiently small, and hence $f_1$ satisfies condition {\rm (i)}. By the definition of $E_1$ in \eqref{eq:defE1} we have 
\[
	 [f_0|_{\Sigma\setminus(S\cup E)}]=(D')^2\prod_{p\in E_1}p,
\]
for some effective divisor $D'\in \Div\big(\Sigma\setminus(S\cup E)\big)$, and hence {\rm (P4)} gives that $[f_1|_{\Sigma\setminus (S\cup E)}] =(DD')^2\prod_{p\in E_1}p^2$. Thus, $f_1$ satisfies condition {\rm (ii)}. 

Finally, to complete the proof we shall find a function $h_1\in \Ocal_\infty(\Sigma|E)$ such that $\wt f=h_1^2f_1$ satisfies the conclusion of the proposition. By Theorem \ref{th:BSR}, for any  $p\in E$ and any $\epsilon>0$,  there is   $F_p\in \Ocal_\infty(\Sigma|\{p\})$, $F_p\not\equiv 0$,  such that  
\begin{equation}\label{eq:Hhp}
	[F_p]\geq D_1[f_1|_S]_\infty\quad \text{ and }\quad \|F_p\|_{S}<\epsilon. 
\end{equation}
In particular, $F_p$ has an effective pole at $p$.  
Choose an integer $m>k+k_0$, where $k$ is the number in the statement of the proposition and $k_0\ge 0$ is the maximum among the zero orders of $f_1$ at the points in $E$, then the function 
\[
  h_1= 1+\sum_{p\in E} F_p^m
\]
satisfies the following conditions in view of \eqref{eq:Hhp}.
\begin{enumerate}[\rm ({P}1)]
\item[\rm ({P}6)]   $\|h_1^2-1\|_{S}<l^2\epsilon^2+2l\epsilon$, where $l$ is the cardinal of $E$.
\smallskip
\item[\rm ({P}7)] $[h_1^2- 1]_0\geq D_1[f_1|_S]_\infty$.
\end{enumerate}
Reasoning as above, it is easily seen that $\wt f=h_1^2f_1$ satisfies {\rm (i)} and {\rm (ii)} provided $\epsilon>0$ is chosen sufficiently small. Since $F_p$ has an effective pole at $p$ for each $p\in E$, we have that $h_1$ has a pole of order at least $m$ at each point $p\in E$. Thus, $\wt f$ has a pole of order at least $2k+k_0\ge k$ at each point $p\in E$, and hence $\wt f$ satisfies {\rm (iii)}. This completes the proof.  
 \end{proof}


\section{Multiplicative sprays 
in $\Sgot^{n-1}_*$}\label{sec:sprays}

Let $M$ be an open  Riemann surface, and let  $S \subset  M$ be a connected, smoothly bounded compact domain.    Also let $E_0$ and $\Lambda$ be a pair of disjoint finite subsets of $\mathring S$. 
Let $n\geq 3$ be an integer, recall the hyperquadrics $\Agot_*^{n-1}$ and $\Sgot_*^{n-1}$ in $\c^n$ and the canonical biholomorphism $\Xi\colon \Agot_*^{n-1}\to \Sgot_*^{n-1}$; see \eqref{eq:nullquadric1}, \eqref{eq:nullquadric}, and \eqref{eq:Thetabiho}.  Consider a full map $f\in \Ocal_\infty(S|E_0,\Agot^{n-1}_*)$, see Definition \ref{de:full}, define
\begin{equation}\label{eq:f0u*}
	  u:=\Xi\circ f\in \Ocal_\infty(S|E_0,\Sgot_*^{n-1}),
\end{equation}
and write $u=(u_1,\ldots,u_n)$. 
Fix an integer $r\geq 0$ and fix the divisors in $\Div(\mathring K)$
\begin{equation}\label{eq:Delta*}
	\Delta=\prod_{p\in E_0\cup \Lambda} p^r \text{\quad and \quad} \Delta_0= \Delta \big(\prod_{i=1}^n [u_i]_\infty\big).
\end{equation}
We denote 
\begin{equation}\label{eq:Omegaeta}
	\Ocal_\Delta(S,u)=\{ v\in \Ocal_\infty\big(S|E_0,\Sgot^{n-1}_*\big)\colon  v_j-u_j\in \Ocal_{\Delta}(S),\, j=1,\ldots,n\},
\end{equation} 
where $v=(v_1,\ldots,v_n)$; see \eqref{eq:OcalDS}. Obviously, $u\in \Ocal_\Delta(S,u)$.
\begin{definition}\label{de:spindata}
We denote by $\Scal_2(\d)$  the space of all maps  $(s_1,s_2,s_3)\in \Ocal(\d,\Sgot^2_*)$  satisfying the following conditions.
\begin{itemize}
\item $s_j$ vanishes nowhere on $\d$ and $s_j(0)=1$ for all $j=1,2,3$.
\smallskip
\item $s_1'(0) s_2'(0)\neq 0$. 
\smallskip
\item Either $s_3\equiv 1$ or $s_3'(0)\neq 0$.
\end{itemize}
\end{definition}
\begin{remark}\label{re:spray*}
For any function $g\in\Ocal_{\Delta_0}(S)$ with $\|g\|_{S}<1$, any map $(s_1,s_2,s_3)\in \Scal_2(\d)$, and any map $v=(v_1,\ldots,v_n)\in \Ocal_\Delta(S,u)$, the map 
\[
\left((s_1\circ g) v_1, (s_2\circ g) v_2,\big(\big( s_3\circ g\big) v_i\big)_{i=3,\ldots,n} \right)\in \Ocal_\Delta(S,u).
\]
 \end{remark}

Take a point $p_0\in \mathring S$ and a skeleton $\{C_1,\ldots,C_l\}$, $l=m+\dim H_1(S,\z)$, of $S$ based at $(p_0,E_0\cup \Lambda)$; see Definition \ref{de:skeleton}.
We choose, as we may,  $\{C_1,\ldots,C_l\}$ to be Jordan arcs or curves such that $C_i\cap C_j=\{p_0\}$ for all $i\neq j\in \{1,\ldots,l\}$ and $C=\bigcup_{j=1}^l C_j\subset \mathring S$  is a strong deformation retract of $S$.
Denote 
\[
	\Ccal^0(C,u)=\{h\in \Ccal^0(C\setminus E_0,\Sgot^{n-1}_*)\colon h-u\in\Ccal^0(C,\c^n)\}.
\]
It turns out that $\Ocal_\Delta(S,u) \subset \Ccal^0(C,u)$.
Fix a  nowhere vanishing holomorphic $1$-form $\theta$  on a neighborhood of $S$ in $M$, and consider the period maps 
\begin{equation}\label{eq:Pcal-map}
	\Pcal  \colon \Ccal^0(C,u)\to (\c^n)^l,\quad \Pcal(h)= \Big( \int_{C_j} (h-u)\theta\Big)_{j=1,\ldots,l},
\end{equation}
\begin{equation}\label{eq:Pcal-map*}
	\Pcal^{1,2}  \colon \Ccal^0(C,u)\to (\c^2)^l,\quad \Pcal^{1,2}(h)= 
	 \left( 
	\begin{array}{c}
 \int_{C_j} (h_1-u_1)\theta\smallskip
	\\ 
 \int_{C_j} (h_2-u_2)\theta
	\end{array}\right)_{j=1,\ldots,l},
\end{equation}
where $h=(h_1,\ldots,h_n)\in \Ccal^0(C,u)$.

In what follows, we shall use the following notation.
Given   points $w=\big((w_{i,j})_{i=1,\ldots,n}\big)_{j=1,\ldots,l}, z=\big((z_{i,j})_{i=1,\ldots,n}\big)_{j=1,\ldots,l}\in (\c^n)^l$, we write
\[
	w\cdot z=\sum_{j=1}^l \sum_{i=1}^n  w_{i,j} z_{i,j}
\]
and denote $\rho\overline\b_{n,l}=\{\zeta \in (\c^n)^l\colon |\zeta|^2=\zeta \cdot \overline \zeta\le\rho^2\}$ for all $\rho>0$.

The main goal of this section   is to show    that we can embed the given map $u \in \Ocal_\infty\big(S|E_0,\Sgot^{n-1}_*\big)$ into some particular  period dominating spray, say of class $\Sgot_*^2$, of maps in $\Ocal_\infty\big(S|E_0,\Sgot^{n-1}_*\big)$. 
%
%
\begin{lemma}\label{le:spray}
Let $M$, $S$, $E_0$, $\Lambda$, $u$, and $\theta$ be as above.
There is  a map $h=\big((h_{i,j})_{i=1,\ldots,n}\big)_{j=1,\ldots,l}   \in  \big(\Ocal_{\Delta_0}(S)^n\big)^l$, $h\not \equiv 0$, for which   the following statement holds true.   
  \begin{enumerate}[\rm ({B}1)]
\item For any number $\rho_0\in ]0,1/\|h\|_S[$ and any map $\sgot=(s_1,s_2,s_3) \in \Scal_2(\d)$, the map   $\Psi_\sgot\colon\rho_0\overline\b_{n,l} \to  \Ocal_\infty\big(S|E_0,\Sgot^{n-1}_*\big)$ given by
\[ \Psi_\sgot(\zeta)=\left(s_1(\zeta\cdot h )u_1, s_2(\zeta\cdot h) u_2,  s_3(\zeta\cdot h)(u_i)_{i=3,\ldots,n} \right)
\]
assumes values in $ \Ocal_\Delta(S,u)$ and is period dominating at $\zeta=0$; the latter meaning  that $\Pcal \circ \Psi_\sgot\colon \rho_0\overline\b_{n,l}  \to (\c^n)^l$ satisfies 
\[
  {\rm rank}\Big( \frac{\partial (\Pcal \circ\Psi_\sgot)}{\partial \zeta}\Big|_{\zeta=0}\Big)=\left \{ \begin{matrix}nl & \text{\quad if \quad}s_3'(0)\neq 0 \text{ (i.e., $s_3\not\equiv 1$)}\smallskip
\\  2 l & \text{\quad if \quad}s_3'(0)=0 \text{ (i.e., $s_3\equiv 1$)}.
\end{matrix}\right.
\] 
\end{enumerate}
Therefore, for any $\epsilon>0$  there is $\rho\in \big]0,1/\|h\|_{S}\big[$ so small that the following conditions are satisfied.
 \begin{enumerate}[\rm ({B}2)]
 \item[\rm ({B}2)]   $\Psi_\sgot(\zeta)\colon S\setminus E_0 \to \Sgot_*^{n-1}\subset\c^n$ is full and $\|\Psi_\zeta-u\|_S<\epsilon$ for all   $ \zeta \in \rho\overline\b_{n,l}$.\hspace{-2mm}
\smallskip
 \item[\rm ({B}3)] $\Pcal \circ \Psi_\sgot\colon \rho\overline\b_{n,l}   \to (\Pcal \circ \Psi)(\rho\overline\b_{n,l})$  is a biholomorphism with $\Pcal( \Psi_\sgot(0))=0$ if $s_3\not\equiv 1$.
\smallskip
\item[\rm ({B}4)] $\Pcal^{1,2} \circ \Psi_\sgot  \colon \rho\overline\b_{n,l}   \to (\c^2)^l$  is a holomorphic submersion satisfying $\Pcal^{1,2} ( \Psi_\sgot(0))=0$ if $s_3\equiv 1$.
 \end{enumerate}
 \end{lemma} 
Note that $u=\Psi_\sgot(0)$ is the core map of the spray $\Psi_\sgot$.
\begin{proof} 
Assume that we have a continuous map 
\begin{equation}\label{eq:F=(F1,F2)*}
	F=\big((f_{i,j})_{i=1,\ldots,n}\big)_{j=1,\ldots,l}  \colon C\to (\c^n)^l
\end{equation}
 satisfying the following conditions for all $j\in\{1,\ldots,l\}$.
\begin{enumerate}[\rm (a)]
\item $\supp (f_{i,j})$ is connected and is contained in $C_j  \setminus(\{p_0\}\cup E_0\cup \Lambda)\subset \mathring S$ for all $i\in\{1,\ldots,n\}$.
\smallskip
\item The compact sets $\supp (f_{i,j})$, $i\in\{1,\ldots,n\}$, are pairwise disjoint.
\end{enumerate}  
Take $\rho_0\in ]0,1/\|F\|_C[$ and $\sgot=(s_1,s_2,s_3)\in \Scal_2(\d)$. Consider the map $\Psi_F\colon \rho_0\overline\b_{n,l}  \to \Ccal^0(C,u)$ given by
\[
	\Psi_F(\zeta)=\left(s_1( \zeta\cdot F )u_1, s_2(\zeta\cdot F) u_2, s_3(\zeta\cdot F) (u_i)_{i=3,\ldots,n} \right).
\]
This map assumes values in  $\Ccal^0(C,u)$  since $\Psi_F(\zeta)=u$ on a neighborhood of $E_0$ in $C$; recall that $\supp (f_{i,j})\subset \mathring S\setminus E_0 $. Furthermore, $\Psi_F$ is Frechet differentiable.

Write $\zeta=\big((\zeta_{i,j})_{i=1,\ldots,n}\big)_{j=1,\ldots,l}\in \rho_0\overline\b_{n,l}$ and 
\[
\Pcal=(\Pcal_1,\ldots,\Pcal_l)\colon\Ccal^0(C,u)\to (\c^n)^l.
\]
 The   Jacobian matrix 
\[
	T_{k,j}(F):=\frac{\partial (\Pcal_k\circ\Psi_F)}{\partial \big(\zeta_{i,j}\big)_{i=1,\ldots, n}}\bigg|_{\zeta=0},\quad k,j\in\{1,\ldots,l\},
\]
is the matrix $0_{n\times n}$ if $k\neq j$, whereas for $k=j$ we have that
\begin{equation}\label{eq:Tjj*} 
		T_{j,j}(F)=   W_j(F)  \cdot {\rm A} ,
\end{equation}
where 
\begin{equation}\label{eq:Wjmatrix}
	W_j(F)= \left(
	\begin{array}{c}
	\displaystyle \int_{C_j}f_{a,j}u_b\theta
	 	\end{array}\right)_{a,b=1\ldots,n}.
\end{equation} 
and ${\rm A}$ is the diagonal matrix of order $n$ whose diagonal entries starting in the upper left corner are $s_1'(0), s_2'(0),s_3'(0),\ldots,s_3'(0)$.
Therefore, if the condition
\begin{equation}\label{eq:rankWj*}
	\det(W_j(F)) \neq 0 \quad \text{for all $j=1,\ldots,l$} 
\end{equation} 
were satisfied, then the block diagonal square matrix of order $nl$
\[
  \frac{\partial (\Pcal \circ\Psi_F)}{\partial \zeta}\big|_{\zeta=0}=
	 \left( 
	\begin{array}{c|c|c}
	T_{1,1}(F) & \cdots & 0
	\\ 
	\hline
	\vdots & \ddots & \vdots
	\\ \hline
	0 & \cdots & T_{l,l}(F)
	\end{array}\right),
\]
would have
\begin{equation}\label{eq:rankvs}
  {\rm rank}\Big( \frac{\partial (\Pcal \circ\Psi_F)}{\partial \zeta}\big|_{\zeta=0}\Big)=nl  \text{\quad if \quad}s_3'(0)\neq 0 \text{ (i.e., $s_3\not\equiv 1$)}
\end{equation}
and
\begin{equation}\label{eq:rankvs12}
 {\rm rank}\Big( \frac{\partial (\Pcal \circ\Psi_F)}{\partial \zeta}\big|_{\zeta=0}\Big)= {\rm rank}\Big( \frac{\partial (\Pcal^{1,2} \circ\Psi_F)}{\partial \zeta}\big|_{\zeta=0}\Big)=2 l \quad \text{if  $s_3\equiv 1$}.
\end{equation}
 We now seek for a continuous map $F$ as in \eqref{eq:F=(F1,F2)*} for which \eqref{eq:rankWj*} is satisfied.  Fix for each $j\in\{1,\ldots,l\}$ pairwise distinct points $p_{1,j},\ldots,p_{n,j}$ in $C_j\setminus\{p_0\}$  such that
\begin{equation}\label{eq:maximal-rank*}
	\big\{u(p_{i,j})\colon i=1,\ldots,n\big\} \quad \text{is a basis of $\c^n$.}
\end{equation}
 The existence of such points is ensured by the fullness of $f$; note that, by \eqref{eq:f0u*} and the analyticity of $u$, the map  $u|_\beta\colon \beta\to\C^n$ is full on any Jordan arc $\beta\subset S$. Next, we consider for each $j\in\{1,\ldots,l\}$ a parameterization $C_j\colon [0,1]\to C_j\subset \mathring  S$ with $C_j(0)=p_0$ and $C_j(1)=p_j$ (here we write $p_j=p_0$ for all $j\in \{m+1,\ldots,l\}$), and denote by $t_{i,j}\in(0,1)$ the point such that $C_j(t_{i,j})=p_{i,j}$ for $i=1,\ldots,n$. Choose a positive number $\delta>0$ so small that   $[t_{i,j}-\delta,t_{i,j}+\delta]\subset (0,1)$ for all $i=1,\ldots,n$, and the arcs  $C_j([t_{i,j}-\delta,t_{i,j}+\delta])\subset \mathring S$, $i=1,\ldots,n$, are pairwise disjoint. We then choose continuous functions $f_{i,j}\colon C\to\c$ with support in $C_j([t_{i,j}-\delta,t_{i,j}+\delta])$, satisfying
\[
	\int_{C_j}f_{i,j}u\theta =\int_{t_{i,j}-\delta}^{t_{i,j}+\delta}f_{i,j}(t) u(C_j(t))\theta(C_j(t),\dot C_j(t)) \,dt
	= u(p_{i,j})
\]
for all $i$ and $j$; recall that $\theta$ vanishes nowhere on $S$.
In view of \eqref{eq:Tjj*}, \eqref{eq:Wjmatrix}, and \eqref{eq:maximal-rank*}, this shows that \eqref{eq:rankWj*} holds true for the map $F$ formed by these functions $f_{i,j}$ (see \eqref{eq:F=(F1,F2)*}), and hence \eqref{eq:rankvs} and \eqref{eq:rankvs12} hold. In particular, for any sufficiently small $0<\rho_1< 1/\|F\|_C$,  
\begin{equation}\label{eq:PsiF*}
\Pcal \circ \Psi_F\colon \rho_1\overline\b_{n,l}   \to (\Pcal \circ \Psi_F)(\rho_1\overline\b_{n,l})
\end{equation}
is a well-defined biholomorphism with $\Pcal (\Psi_F(0))=0$ if $s_3\not\equiv 1$, and  
\begin{equation}\label{eq:PsiF*12}
\Pcal^{1,2} \circ \Psi_F\colon \rho_1\overline\b_{n,l}   \to (\c^2)^l
\end{equation}
is a holomorphic submersion with $\Pcal^{1,2} (\Psi_F(0))=0$ if $s_3\equiv 1$.

Fix $r_0\in \n$ with
\begin{equation} \label{eq:r0}
r_0\geq r+\sum_{p\in E_0}\sum_{i=1}^n{\rm Ord}_p(u_i),
\end{equation}
where ${\rm Ord}_p(\cdot) $ means pole order at $p\in E_0$ and $r$ is the integer in \eqref{eq:Delta*}. Since $C$ is a strong deformation retract of $S$, the Runge-Mergelyan theorem with jet-interpolation (see e.g.\ \cite[Theorems 3.8.1 and 5.4.4]{Forstneric2017}) ensures that we may approximate each $f_{i,j}$ uniformly on $C$ by a function $h_{i,j}\in\Ocal(S)$, $h_{i,j}\not\equiv 0$, vanishing to order $r_0$ at every point of $\Lambda\cup E_0$;  observe that $f_{i,j}\equiv 0$ on a neighborhood of $\Lambda\cup E_0$. 
It follows that $h=\big((h_{i,j})_{i=1,\ldots,n}\big)_{j=1,\ldots,l}\in \big(\Ocal_{\Delta_0}(S)^n\big)^l$; see \eqref{eq:OcalDS} and \eqref{eq:Delta*}. If we define $\Psi_\sgot$ as in the statement of the lemma for this map $h$ and any positive $\rho_0<1/\|h\|_{S}$, it turns out that  $\Psi_\sgot(\zeta)\in\Ocal_\Delta(S,u)$ for all $\zeta\in \rho_0\overline\b_{n,l}$; see Remark \ref{re:spray*}. Assuming that  $h$ is close enough to $F$ on $C$ and chossing $\rho<\min (\rho_1,1/\|h\|_{S})$ sufficiently small, then {\rm (B1)} and {\rm (B2)}  hold; see \eqref{eq:Wjmatrix}, \eqref{eq:rankvs}, and  \eqref{eq:rankvs12} and recall that $u=\Psi_\sgot(0)$ is full. Furthermore,  in view of   \eqref{eq:PsiF*} and  \eqref{eq:PsiF*12},    {\rm (B3)} and  {\rm (B4)} are satisfied as well provided that $\rho<\min \{\rho_1,1/\|h\|_{S}\}$ is chosen small enough.
 \end{proof}


\section{Runge's theorem for complete minimal surfaces of finite total curvature}\label{sec:Runge}

We now prove the following more precise version of Theorem \ref{th:intro1}. Recall the notation in Section \ref{sec:minimal}; in particular, see \eqref{eq:GCCMI} and \eqref{eq:CCMI}.

\begin{theorem}\label{th:MT} 
Let $\Sigma$ be a compact Riemann surface (without boundary), $E\subset \Sigma$ be a nonempty finite subset, and $S=K\cup\Gamma\subset \Sigma\setminus E$ be an admissible subset (see Definition \ref{def:admissible}) that is Runge in $\Sigma\setminus E$. Also let $E_0, \Lambda$ be a pair of disjoint finite subsets of $\mathring S$ and let $n\geq 3$ be an integer.

For any  $\hat X=(X,f\theta)\in \GCMI_\infty(S|E_0,\R^n)$, any group homomorphism $\pgot\colon H_1(\Sigma\setminus (E_0\cup E),\z)\to \r^n$ with $\pgot|_{H_1(S\setminus E_0,\z)}=\Flux_{\hat X}$, any number $\epsilon>0$, and any integer $r\geq 0$, there is a conformal minimal immersion $Y\colon \Sigma\setminus (E_0\cup E)\to \r^n$ satisfying the following conditions.
\begin{enumerate}[{\rm (i)}]
\item $Y$ is complete and of finite total curvature.
\smallskip
\item $Y-X$ extends to $S$ as a continuous map and $\|Y-X\|_{S}<\epsilon$.
\smallskip
\item $Y-X$ vanishes at least to order $r$ at every point of $\Lambda\cup E_0$.
\smallskip
\item $\Flux_Y=\pgot$.
\end{enumerate}
\end{theorem}
\begin{proof}
%
%
We begin with the following.
\begin{claim}\label{cl:conazo}
There are a Runge  admissible subset $S'=K'\cup \Gamma'\subset \Sigma\setminus  E$ (see Def.\ \ref{def:admissible}) and an immersion $\hat X'=(X',f'\theta)\in\GCMI_\infty(S'|E_0,\r^n)$ satisfying the following requirements.
\begin{enumerate}[{\rm (a)}]
\item  $S\subset S'$ and $S'$ is a  strong deformation retract of $\Sigma\setminus  E$.
\smallskip
\item $K'\neq\varnothing$ and every component of $\Gamma'$ intersects $K'$.
\smallskip
\item $K$ is a union of components of $K'$.
\smallskip
\item $X'|_{K\setminus E_0}= X|_{K\setminus E_0}$ and $\hat X'|_{S\setminus E_0} \approx \hat X$.
\smallskip
\item $\Flux_{\hat X'}=\pgot$.
\end{enumerate}
\end{claim}
%
%
\begin{proof} 
By elementary topological arguments, since $\Sigma\setminus E$ has finite topology there is a Runge admissible subset $S'=K'\cup \Gamma'\subset \Sigma\setminus  E$  satisfying {\rm (a)}, {\rm (b)}, and {\rm (c)}. Such an $S'$ can be found such that $K'\setminus K\neq \varnothing$,  $K'\setminus K$   consists of   pairwise disjoint closed discs, and every component of $K'$ intersects at most one component of $\Gamma$.  

Let $W_1$ denote the union of $S$ and all the components of $K'\setminus K$ intersecting $\Gamma$,  and notice that $S$ is a strong deformation retract of $W_1$. 
Choosing the components (closed discs) of  $(W_1\cap K')\setminus K$ sufficiently small (say, so small that $f$ is close to  a locally constant map on $(W_1\cap  K'\setminus K)\cap \Gamma$), we can extend $X|_K$ to an immersion  $\hat X_1\in \GCMI(W_1, \r^n)$ such that $\hat X_1$ is close to flat on $W_1\cap K'\setminus K$, $\hat X_1$ is close to $\hat X$ on $\Gamma$, and $\Flux_{\hat X_1}=\Flux_{\hat X}$. Indeed, we can for instance choose $\hat X_1=\hat X$ outside a small neighborhood $V$ of $W_1\cap K'\setminus K$ in $W_1$ and to be a slight modification of $\hat X$ on $V\setminus ((W_1\cap K')\setminus K)\subset\Gamma$ which ensures that $\hat X_1\in \GCMI(W_1, \r^n)$ and the condition on the flux; we use \cite[Lemma 3.3]{AlarconCastro-Infantes2019APDE} for this construction.

Next, set $W_2=W_1\cup K'$ and extend $\hat X_1$  to a generalized conformal minimal immersion $\hat X_2\in \GCMI(W_2, \r^n)\cap \CMI(K'\setminus W_1)$. Finally, we obtain an immersion  $\hat X'\in\GCMI(S', \r^n)$ satisfying conditions {\rm (d)} and {\rm (e)} by extending
  $\hat X_2$ to the arcs in $\Gamma' \setminus (\mathring W_2\cup \Gamma)$ in such a way that  $\Flux_{\hat X'}=\pgot$; we use again \cite[Lemma 3.3]{AlarconCastro-Infantes2019APDE}.  
 \end{proof}
%
%

Up to replacing $(S,\hat X)$ by $(S',\hat X')$, and then using Theorem \ref{th:Merg}, we can assume that $S$ is a strong deformation retract of $\Sigma\setminus E$, $X\in\CMI_\infty(S|E_0,\R^n)$, and 
\begin{equation}\label{eq:f}
	\text{$f=2\di X/\theta\colon S\setminus E_0\to\C^n$\quad is full.}
\end{equation}
Furthermore, since $X$ extends to a neighborhood of $S$ as a conformal minimal immersion, we can also assume without loss of generality that 
$\Gamma=\emptyset$   and $S=K$ is a connected, smoothly bounded, compact domain.

We  assume without loss of generality that the finite set $\Lambda \subset\mathring S\setminus E_0\neq\varnothing$ is nonempty and write 
\begin{equation}\label{eq:p1...pm}
	E_0\cup \Lambda=\{p_1,\ldots,p_m\}.
\end{equation}
Fix a point $p_0\in \mathring S\setminus \{p_1,\ldots,p_m\}$ and choose a skeleton  $\{C_ 1,\ldots,C_l\}$, $l=m+\dim H_1(S,\z)$, of $S$ based at $(p_0, E_0\cup \Lambda)$; see Definition \ref{de:skeleton}. We choose the skeleton, as we may since $S=K$ is connected, such that
	\[
	C_i\cap \big(\bigcup_{i\neq j=1}^l C_j\big)=\{p_0\}\quad \text{for all $i\in \{1,\ldots,l\}$.}
	\]
It turns out that $C:=\bigcup_{j=1}^lC_j$ is a Runge subset of $\Sigma\setminus E$ that is a strong deformation retract of $\Sigma\setminus E$.

%
%
Recall the following classical result; we include a proof for completeness.
\begin{claim}\label{eq:theta0exists}
There is a $1$-form $\theta_0\in\Omega_\infty(\Sigma|E)$ vanishing nowhere on $S$ and having $[\theta_0]_\infty\geq \prod_{p\in E} p$.
\end{claim}
\begin{proof}
Fix any $1$-form $\tau\in\Omega_\infty(\Sigma|E)$. By the classical Weierstrass theorem, there is $h\in\Ocal(\Sigma\setminus E)$ with $[h]=[\tau|_S]$. Set $E_1=\supp([h])\subset S$. By Proposition \ref{pro:Royden0} and Hurwitz's theorem, there exists a function $g\in\Ocal_\infty(\Sigma|E\cup E_1)$ close to $1/h$ on $S$ with $[g|_S][h]=1$ and $[g \tau]_\infty\geq \prod_{p\in E} p$. It suffices to choose $\theta_0=g\tau$.
\end{proof}
 
Fix $\theta_0\in\Omega_\infty(\Sigma|E)$ as in Claim \ref{eq:theta0exists}. 
Define 
\begin{equation}\label{eq:f0u}
	f_0:=f \theta/\theta_0\in \Ocal_\infty(S|E_0,\Agot_*^{n-1}),\quad u:=\Xi\circ f_0\in \Ocal_\infty(S|E_0,\Sgot_*^{n-1}),
\end{equation}
and write $u=(u_1,\ldots,u_n)$; see \eqref{eq:f} and \eqref{eq:Thetabiho}.
Set
\begin{equation}\label{eq:Delta}
	\Delta=\prod_{p\in E_0\cup \Lambda} p^r \text{\quad and \quad} \Delta_0= \Delta \big(\prod_{i=1}^n [u_i]_\infty\big).
\end{equation}
(Here, $r$ is the integer given in the statement of Theorem \ref{th:MT}.) 
Denote
\begin{equation}\label{eq:Thetatheta0}
	\Theta: =\supp([\theta_0]_0)\subset \Sigma\setminus (S\cup E).
\end{equation}

The next stage in the proof is to approximate $u$, uniformly on $S$, by a certain map $\hat u=(\hat u_1,\ldots,\hat u_n)\in \Ocal_\infty(\Sigma|E\cup E_0\cup \Theta,\Sgot^{n-1}_*)$ such that $\hat u \theta_0$ vanishes nowhere on $\Sigma\setminus (E_0\cup E)$. For, we proceed in two steps: we first approximate $u_1$ by a function $\hat u_1\in \Ocal_\infty(\Sigma|E\cup E_0\cup \Theta)$ and, after that, we approximate $(u_3,\ldots,u_n)$ by a suitable map $(\hat u_3,\ldots,\hat u_n)\in\Ocal_\infty(\Sigma|E\cup E_0\cup \Theta)^{n-2}$; the function $\hat u_2\in\Ocal_\infty(\Sigma|E\cup E_0 \cup \Theta)$ approximating $u_2$ will then come forced by the requirement that $\hat u$ assumes values in $\Sgot^{n-1}_*$.

We begin with some preparations. We assume, as we may up to slightly enlarging $S$ if necessary, that $u_i$ vanishes nowhere on $bS=S\setminus \mathring S$ for all $i\in\{1,\ldots,n\}$; recall that $u_i\in\Ocal_\infty(S|E_0)$.  Consider the following effective divisor
\begin{equation}\label{eq:masterdivisor}
Z= \big(\prod_{i=1}^n [u_i]^2_0 [u_i ]^2_\infty\big)\in \Div(S)
\end{equation}
whose support lies in $\mathring S$.
%

%
%

Fix a number $\epsilon_0>0$ to be specified later. 

Consider the map $\sgot\colon \c\to\Sgot_*^2$ given by
\[
	\sgot:=\big((1+z)^2,(1+z/2)^2,(1+z) (1+z/2)\big);
\]  
note that $\sgot|_\d\in  \Scal_2(\d)$, see Definition \ref{de:spindata}. 
 Let $h=\big( (h_{i,j})_{i=1,\ldots,n}\big)_{j=1,\ldots,l} \in \big(\Ocal_{\Delta_0}(S)^n\big)^l$ be given by
\[
\Psi_\sgot(\zeta)=\left((1+\zeta\cdot h )^2 u_1, (1+  \zeta\cdot h/2)^2  u_2,  (1+\zeta\cdot h)(1+\zeta\cdot h/2)(u_i)_{i=3,\ldots,n} \right),
\]
 and $\rho>0$ be the objects provided by Lemma \ref{le:spray} applied to the map $u\in \Ocal_\infty(S|E_0,\Sgot^{n-1}_*)$   in \eqref{eq:f0u},  $\sgot|_{\d}$, the divisors in \eqref{eq:Delta} (compare with \eqref{eq:Delta*}), the $1$-form $\theta_0$, and the number $\epsilon_0$. (The lemma is applied on a small open neighborhood of $S$ in $M$ where $\theta_0$ vanishes nowhere.)

We shall first deal with the first component $u_1$ of $u$. For, fix a number $\epsilon_1>0$ to be specified later. 

Take pairwise disjoint closed discs  $U_q\subset \Sigma\setminus (S\cup E)$, $q\in \Theta$,  with $q\in \mathring U_q$ for all $q\in \Theta$, and call $U=\bigcup_{q\in \Theta} U_q$. Take any function $u_1^*\in \Ocal_\infty\big(S\cup U|E_0\cup \Theta\big)$ such that 
\begin{equation}\label{eq:u1*}
\text{$u_1^*|_S=u_1$\quad and\quad   $[u_1^*|_{U}] [\theta_0]_0 =1$.}
\end{equation}
By Proposition \ref{pro:Royden0} and Hurwitz's theorem, there is $\hat u_1\in \Ocal_\infty(\Sigma|E\cup E_0  \cup \Theta)$ satisfying the following conditions.
\begin{enumerate}[\rm ({C}1)]
\item $\hat u_1-u_1^*\in\Ocal(S\cup U)$ and $\|\hat u_1-u_1^*\|_{S\cup U}<\epsilon_1$.
\smallskip
\item $\hat u_1|_{S\cup U}-u_1^*\in \Ocal_{Z\Delta  [\theta_0]_0^2}(S\cup U)$, see \eqref{eq:Delta},  \eqref{eq:masterdivisor}, and \eqref{eq:OcalDS}.
\smallskip
\item $[\hat u_1|_{\Sigma\setminus E}]= D_1^2 [u_1^*]$
  for some effective divisor $D_1 \in \Div(\Sigma\setminus (S\cup U \cup E))$; in particular, $[\hat u_1|_{S\cup U}]=[u_1^*]$. By \eqref{eq:u1*}, it turns out that  $\hat u_1\theta_0\in \Omega_\infty(\Sigma|E_0\cup E)$ and  $\hat u_1\theta_0$   has no zeros on $U$.
\smallskip
\item $[\hat u_1\theta_0]_\infty \ge \prod_{p\in E}p$.
 \end{enumerate}
 Moreover, {\rm (C1)} and {\rm (C2)} give that $\big(\frac{u_1}{\hat u_1|_{S}}-1\big) u_2 \in \Ocal(S)$, and we may assume that
\begin{equation}\label{eq:C5}
	\Big\| \Big(\frac{u_1}{\hat u_1|_{S}}-1\Big) u_2 \Big\|_{S}<\epsilon_1.	
\end{equation}

Take pairwise disjoint closed discs  $T_q\subset \Sigma\setminus \big(S\cup E\cup U\big)$, $q\in \supp(D_1)$,  with $q\in \mathring T_q$ for all $q\in \supp(D_1)$, and  call $T=\bigcup_{q\in \supp(D_1)} T_q$; see {\rm (C3)}. Note that
\begin{equation}\label{eq:D1cuad}
D_1^2=[\hat u_1|_{T}]_0=[\hat u_1|_{\Sigma\setminus(S\cup E)}]_0. 
\end{equation}

We shall now deal with the last $n-2$ components of $u$. For, consider functions $v_3,\ldots,v_n$  in $\Ocal_\infty(S\cup T|E_0)$ such that 
\begin{equation}\label{eq:u3un}
\text{$v_i|_S=u_i$ for all $i=3,\ldots,n$, and }   \big(\sum_{i=3}^n v_i^2\big)\big|_{T} =  \hat u_1|_{T};
\end{equation}
note that, in view of \eqref{eq:D1cuad}, such extensions exist even for $n=3$. (One may for instance choose $v_3|_{T}=\sqrt{\hat u_1|_{T}}$ and $v_i|_{T}\equiv 0$ for all $i\geq 4$.)
Likewise, let $v_1, v_2\in \Ocal_\infty(S\cup T|E_0)$ be the functions given by 
\begin{equation}\label{eq:v1v2}
\text{$v_i|_S=u_i$,  $i=1,2$,\quad $v_1|_{T}=\hat u_1|_{T}$,\quad and\quad $v_2|_{T}\equiv1$.}
\end{equation} 

Fix $\epsilon_2$, $0<\epsilon_2<\epsilon_1$, to be specified later. 

Proposition \ref{pro:Royden0} and Hurwitz's theorem provide a map $(\hat u_3,\ldots,\hat u_n)\in \Ocal_\infty(\Sigma|E\cup E_0)^{n-2}$ satisfying the following conditions for all $i=3,\ldots,n$.
\begin{enumerate}[\rm ({D}1)]
\item $\hat u_i-v_i\in\Ocal(S\cup T)$ and $\|\hat u_i-v_i\|_{S\cup T}<\epsilon_2$. 
\smallskip
\item $\hat u_i|_{S\cup T}-v_i\in \Ocal_{Z\Delta D_1^4}(S\cup T)$.
\smallskip
\item   $[\hat u_i|_{S}]=[u_i]$; see \eqref{eq:v1v2}.
\end{enumerate}

%
%
We claim that if $\epsilon_2>0$ is chosen sufficiently small (in terms of the fixed but yet unspecified number $\epsilon_1>0$), then the function 
\[
	\hat u_2:=\frac{\sum_{i=3}^n \hat u_i^2}{\hat u_1}
\] 
satisfies the following conditions.
\begin{enumerate}[\rm ({E}1)]
\item $\hat u_2\in  \Ocal_\infty(\Sigma|E_0\cup E)$.
\smallskip
\item  $\hat u_2-v_2\in\Ocal(S\cup T)$ and $\|\hat u_2-v_2\|_{S\cup T}<\epsilon_1$.
\smallskip
\item  $\hat u_2|_{S\cup T}-v_2\in \Ocal_{\Delta[u_2]_0}(S\cup T)$.
\smallskip
\item  $[\hat u_2|_{S\cup T}]=[\hat u_2|_{S}]=[v_2]=[u_2]$.
\end{enumerate}
Indeed,
we shall first check {\rm (E1)}.
Properties \eqref{eq:u3un}, {\rm (D2)}, and {\rm (D3)}  ensure for each $i=3,\ldots,n$ that 
 \begin{eqnarray*}
 \big[\hat u_i^2\big|_{S\cup T}-v_i^2\big] & = &  \big[ \hat u_i\big|_{S\cup T}-v_i\big]\big[ \hat u_i\big|_{S\cup T}+v_i \big]
 \\
 & \ge & Z\Delta D_1^4 [u_i]^{-1}_\infty \;\stackrel{\eqref{eq:masterdivisor}}{\ge} \; \Delta D_1^4 [u_1]_0^2[u_2]_0,
\end{eqnarray*}
and hence, by \eqref{eq:f0u}, \eqref{eq:masterdivisor}, \eqref{eq:u3un}, and \eqref{eq:v1v2},
\[
 \big[(\hat u_1 \hat u_2)\big|_{S\cup T}-v_1 v_2\big] = \big[\sum_{i=3}^n\hat u_i^2|_{S\cup T}-\sum_{i=3}^n v_i^2\big] 
  \ge   \Delta D_1^4 [u_1]_0^2[u_2]_0.
\]
In view of \eqref{eq:u1*}, {\rm (C3)}, and \eqref{eq:u3un}, we have $[\hat u_1|_{S\cup T}]_0=[u_1]_0D_1^2$, and so 
\begin{equation}\label{eq:u2div}
\Big[\hat u_2\big|_{S\cup T}- \frac{v_1v_2}{\hat u_1|_{S\cup T}}\Big]\ge  \Delta D_1^2[u_1]_0[u_2]_0.
\end{equation}
It turns out that 
\begin{equation}\label{eq:u2ST}
	\hat u_2|_{S\cup T}-\displaystyle \frac{v_1v_2}{\hat u_1|_{S\cup T}}\in  \Ocal(S\cup T).
\end{equation}
 Moreover, \eqref{eq:u1*}, \eqref{eq:v1v2}, and {\rm (C3)} ensure that $v_1/\hat u_1|_{S\cup T}\in \Ocal (S\cup T)$. Since $v_2\in\Ocal_\infty(S\cup T|E_0)$, we obtain that  $v_1v_2/\hat u_1|_{S\cup T}\in \Ocal_\infty(S\cup T|E_0)$, and so, in view of \eqref{eq:u2ST}, $\hat u_2|_{S\cup T}\in\Ocal_\infty(S\cup T|E_0)$ as well. On the other hand, since $\hat u_i\in \Ocal_\infty(\Sigma|E\cup E_0)$ 
 for all $i\ge 3$ and $\hat u_1$ vanishes nowhere off $S\cup T\cup E$ (see \eqref{eq:D1cuad}), we infer that $\hat u_2\in  \Ocal(\Sigma\setminus (S\cup T\cup E))$. This proves {\rm (E1)}.

In order to check {\rm (E2)} and {\rm (E3)}, we use  \eqref{eq:u1*}, {\rm (C2)}, {\rm (C3)}, and \eqref{eq:v1v2} to infer that
\[
 \Big[\frac{v_1}{\hat u_1|_{S\cup T}}-1\Big] \ge Z\Delta [u_1]_0^{-1} \stackrel{\eqref{eq:masterdivisor}}{\ge} \Delta [u_1]_0 [u_2]_\infty.
\]
Together with \eqref{eq:v1v2} and \eqref{eq:u2div}, we obtain that
\[
	[ \hat u_2\big|_{S\cup T}-v_2 ] = \Big[\Big(\hat u_2\big|_{S\cup T}- \frac{v_1v_2}{\hat u_1|_{S\cup T}}\Big)+v_2\Big( \frac{v_1}{\hat u_1|_{S\cup T}}-1 \Big)  \Big] \ge \Delta [u_1]_0[u_2]_0.
\]
This shows {\rm (E3)} and the first part of {\rm (E2)}; the second part of {\rm (E2)} is ensured by {\rm (D1)}, {\rm (D2)}, \eqref{eq:masterdivisor}, \eqref{eq:C5}, \eqref{eq:u3un}, and \eqref{eq:v1v2} whenever that $\epsilon_2>0$ is chosen sufficiently small. Finally, condition {\rm (E4)} follows from {\rm (E2)}, {\rm (E3)}, and Hurwitz's Theorem provided that $\epsilon_2>0$ is small enough.

Set 
\[
	\hat u:=(\hat u_1,\ldots,\hat u_n)\in \Ocal_\infty(\Sigma|E_0\cup E\cup \Theta)\times \Ocal_\infty(\Sigma|E_0\cup E)^{n-1}.
\]
Summarizing, the following conditions hold true.
\begin{enumerate}[\rm ({F}1)]
\item $\hat u \theta_0\in \Omega_\infty(\Sigma|E_0\cup E)^n$ vanishes nowhere on $\Sigma\setminus (E\cup E_0)$, and hence $\hat u\in \Ocal_\infty(\Sigma|E_0\cup E\cup \Theta,\Sgot^{n-1}_*)$.
\smallskip
\item    $\hat u-u \in\Ocal(S,\c^n)$ and $\|\hat u-u\|_{S}<\sqrt{n} \epsilon_1$.
\smallskip
\item  $\hat u_i|_S-u_i\in \Ocal_\Delta(S)$ for all $i\in\{1,\ldots,n\}$, i.e.,  $\hat u\in \Ocal_\Delta(S,u)\cap \Ocal_\infty(S|E_0)^n$.\hspace*{-1mm}
\smallskip
\item $\hat u_1\theta_0$ (and hence $\hat u\theta_0$) has an effective pole at each point in $E$.
\smallskip
\item $[\hat u_i|_S]=[u_i]$ for all $i\in\{1,\ldots,n\}$.
\end{enumerate}
Indeed, to check {\rm (F1)} recall that $\hat u_1 \theta_0$ vanishes nowhere on $\Sigma\setminus (S\cup T\cup E)$, by {\rm (C3)}. On the other hand, since $\theta_0$ has no zeros in $S\cup T$, \eqref{eq:v1v2} and {\rm (E4)} ensure that $\hat u_2\theta_0$ vanishes nowhere on $T$. Finally, since $u\theta_0$ has no zeros on $S$, {\rm (C3)}, {\rm (D3)}, and {\rm (E4)} imply that $\hat u\theta_0$ vanishes nowhere on $S$. This shows the first part of {\rm (F1)}; the second one then follows from \eqref{eq:Thetatheta0} and the definition of $\hat u_2$. On the other hand, {\rm (F2)}, {\rm (F3)}, and {\rm (F5)} follow straightforwardly from the above properties, whereas {\rm (F4)} is implied by {\rm (C4)}. 

%
%
 
We next deal with the period problem. For, we apply Proposition \ref{pro:Royden0} in order to approximate   $h$   by a map 
 \[
 	\hat h=\big( (\hat h_{i,j})_{i=1,\ldots,n}\big)_{j=1,\ldots,l}\in  \big(\big(\Ocal_{\Delta_0}(S)\cap\Ocal_\infty(\Sigma|E)\big)^n\big)^l,
\]
 satisfying the following conditions.
\begin{enumerate}[\rm ({G}1)]
\item  $\|\hat h-h\|_{S}<\epsilon_1$.
\smallskip
\item $\hat h$ vanishes everywhere on $\big(\bigcup_{i=1}^n \supp([\hat u_i\theta_0]_0)\big)\setminus S$.
\smallskip
\item $[\hat h_{i,j}]_\infty=\prod_{p\in E} p^{m_{i,j}(p)}$ for all $i=1,\ldots,n$,  $j=1,\ldots,l$, where $\{m_{i,j}(p)\colon i=1,\ldots,n,\, j=1,\ldots,l\}$ are pairwise distinct natural numbers for each $p\in E$.
\end{enumerate}
To ensure {\rm (G3)} we fix an ordering in the set $\{1,\ldots,n\}\times\{1,\ldots,l\}$ and apply Proposition \ref{pro:Royden0} recursively in order to guarantee that the map $\{1,\ldots,n\}\times\{1,\ldots,l\}\ni (i,j)\mapsto m_{i,j}(p)$ is strictly increasing for each $p\in E$.

For each $\zeta \in (\c^n)^l$, consider the function
\[
 \hat \Psi_\sgot(\zeta)=\left((1+\zeta\cdot \hat h )^2\hat u_1, (1+  \zeta\cdot \hat h/2)^2 \hat u_2,  (1+\zeta\cdot \hat h)(1+\zeta\cdot \hat h/2)(\hat u_i)_{i=3,\ldots,n} \right),
\]
which obviously lies in $\Ocal_\infty(\Sigma|E\cup E_0\cup \Theta)^n$ by {\rm (F1)}. 

In view of {\rm (F2)} and {\rm (G1)}, if $\epsilon_1>0$ is small then   $\hat \Psi_\sgot(\zeta)$ is  close   to $\Psi_\sgot(\zeta)$ on $S$ uniformly on $\zeta\in \rho\overline \b_{n,l}$, and hence $ \Pcal \circ \hat \Psi_\sgot $ is    close  to $\Pcal \circ   \Psi_\sgot$ on   $ \rho\overline \b_{n,l}$, where $\Pcal$ is the period map \eqref{eq:Pcal-map} with $\theta_0$ in the role of $\theta$; i.e.,
\begin{equation}\label{eq:Pcal-map-0}
		\Pcal  \colon \Ccal^0(C,u)\to (\c^n)^l,\quad \Pcal(h)= \Big( \int_{C_j} (h-u)\theta_0\Big)_{j=1,\ldots,l}.
\end{equation}
Thus, the following assertions  hold provided that  $\epsilon_1>0$ is  small enough (in terms, in particular, of the fixed but still to be specified number $\epsilon_0>0$).
\begin{enumerate}[\rm ({H}1)]
\item $\hat \Psi_\sgot(\zeta) \in \Ocal_\Delta(S,u)$ for all  $\zeta \in \rho\overline\b_{n,l}$; see \eqref{eq:Omegaeta}. Use  {\rm (F3)}, the fact that $\hat h\in(\Ocal_{\Delta_0}(S)^n)^l$, that we can assume that $\rho<1/\|\hat h\|_{S}$ by {\rm (G1)}, and Remark \ref{re:spray*}.
\smallskip
\item The map $\Pcal \circ \hat \Psi_\sgot\colon \rho\overline\b_{n,l}   \to (\Pcal \circ \hat \Psi_\sgot)(\rho\overline\b_{n,l})$
is a biholomorphism with $0\in (\Pcal \circ \hat \Psi_\sgot)(\rho\b_{n,l})$; use  Lemma \ref{le:spray}-{\rm (B3)}, {\rm (F1)}, {\rm (F2)}, {\rm (G1)}, {\rm (G2)},  and the Cauchy estimates.
\smallskip
\item $\hat \Psi_\sgot(\zeta) \in \Ocal_\infty(\Sigma|E\cup E_0\cup\Theta,\Sgot^{n-1}_*)$ and is full for all $ \zeta \in \rho\overline\b_{n,l}$; use  Lemma \ref{le:spray}-{\rm (B1)},{\rm (B2)}, {\rm (F2)}, and {\rm (G1)}. To check that $\hat \Psi_\sgot(\zeta)$ vanishes nowhere in $\Sigma\setminus E\cup E_0\cup\Theta$, note that neither $\hat u$ nor $\sgot(\zeta\cdot\hat h)$ vanish anywhere there and take into account {\rm (G2)}.   
\smallskip
\item $\|\hat \Psi_\sgot(\zeta) -u\|_{S}<\epsilon_0$  for all  $ \zeta \in \rho\overline\b_{n,l}$; use Lemma \ref{le:spray}-{\rm (B2)}, {\rm (F2)}, and {\rm (G1)}.
 \end{enumerate}
On the other hand, conditions {\rm (F1)} and {\rm (G2)} ensure that, for any $\zeta\in (\c^n)^l$, the vectorial $1$-form 
\begin{equation}\label{eq:421}
\hat\Psi_\sgot(\zeta)\theta_0\in \Omega_\infty(\Sigma|E\cup E_0)^n
\end{equation} 
and vanishes nowhere on $\Sigma\setminus (S\cup E)$. Thus, by {\rm ({F}1)} and {\rm ({H}1)}, we have
\begin{equation}\label{eq:F1Psi}
\hat\Psi_\sgot(\zeta)\theta_0
\text{  has no zeros on }\Sigma\setminus (E_0\cup E),\text{ for all }  \zeta \in \rho\overline\b_{n,l}.
\end{equation} 
 Furthermore,  by {\rm (F4)} and {\rm (G3)}, 
\begin{equation}\label{eq:Psithehatpole}
\text{$\hat\Psi_\sgot(\zeta)\theta_0$ has an effective pole at each point $p\in E$, for all $\zeta \in (\c^n)^l$}.
\end{equation}

Denote by $\zeta_0\in \rho\b_{n,l}$ the point such that 
\begin{equation}\label{eq:zeta0}
	\Pcal(\hat\Psi_\sgot(\zeta_0))=0
\end{equation}
(see {\rm (H2)}) and define
\[
	\hat f=(\hat f_1,\ldots,\hat f_n)=\Xi^{-1}\circ\hat\Psi_\sgot(\zeta_0)\in\Ocal_\infty(\Sigma|E\cup E_0\cup\Theta,\Agot^{n-1}_*),
\] 
where $\Xi$ is the linear biholomorphism \eqref{eq:Thetabiho}; see  {\rm (H3)}.  
The following assertions are satisfied.
\begin{enumerate}[\rm (\text{$\rm I$}1)]
\item $\hat f\theta_0\in\Omega_\infty(\Sigma|E\cup E_0)^n$ and vanishes nowhere on $\Sigma\setminus (E_0\cup E)$, see \eqref{eq:421} and \eqref{eq:F1Psi}.
\smallskip
\item $\hat f\theta_0$ has an effective pole at each point $p\in E$, see  \eqref{eq:Psithehatpole}.
\smallskip
\item $[(\hat f_i\theta_0)|_S-f_i\theta]\ge \Delta$ for all $i\in\{1,\ldots,n\}$, where the map $f=(f_1,\ldots,f_n)$ and the $1$-form $\theta$ are those given in the statement of Theorem \ref{th:MT}; see  \eqref{eq:f0u} and  {\rm (H1)}. In particular, for each $p\in E_0$ there is $i\in\{1,\ldots,n\}$ such that $\hat f_i\theta_0$ has an effective pole at $p$. 
\smallskip
\item $\int_{C_j}(\hat f\theta_0-f\theta)=0$ for all $j\in\{1,\ldots,l\}$, and hence $\hat f\theta_0-f\theta$ is exact on $S$; see \eqref{eq:f0u}, \eqref{eq:Pcal-map-0}, and \eqref{eq:zeta0}.
\smallskip
\item $\|\hat f-f\theta/\theta_0\|_{S}<2 \epsilon_0$; use \eqref{eq:Thetabiho}, \eqref{eq:f0u}, and {\rm (H4)}.
\end{enumerate}

The proof of the theorem is now completed as follows. Since  $S$ is a strong deformation retract of $\Sigma\setminus E$ and $\{C_1,\ldots,C_l\}$ is an skeleton of $S$ based at $(p_0,E_0\cup \Lambda)$ (see Definition \ref{de:skeleton}), \eqref{eq:Pcal-map-0} and {\rm (I1)}--{\rm (I4)} ensure that  
\[
	Y\colon \Sigma\setminus (E_0\cup E)\to \r^n,\quad Y(p)=\Re\int_{p_0}^p \hat f \theta_0
\]
is a well defined complete conformal minimal immersion of FTC with $\Flux_Y=\Flux_X=\pgot$. Moreover,   \eqref{eq:Pcal-map-0}, {\rm (I3)}, {\rm (I4)}, and \eqref{eq:Delta} guarantee that 
$Y-X$ vanishes at least to order $r$ at every point of $\Lambda\cup E_0$. Finally,  {\rm (I5)} and the compactness of $S$ ensure that  $\|Y-X\|_{S}<\epsilon$ provided that $\epsilon_0$ is chosen small enough from the beginning. This concludes the proof of the theorem.
\end{proof}


\section{Mittag-Leffler's theorem for minimal surfaces}\label{sec:Mittag-Leffler}

In this section we prove the following Mittag-Leffler type theorem for conformal minimal surfaces, including approximation and interpolation, which is a more precise version of Theorem \ref{th:intro2}.
\begin{theorem}\label{th:M-Lgen}
Let $M$ be an open Riemann surface, $ A\subset M$ be a  closed discrete subset,  $U\subset M$ be a locally connected, closed neighborhood of $A$ whose connected components are all Runge admissible compact subsets in $M$, and  $X:U\setminus A\to\r^n$ $(n\ge 3)$  be a map such that $X|_{W\setminus A}\in \GCMI_\infty(W| A,\r^n)$ for all components $W$ of $U$. 
Then for any $\Lambda\subset \mathring U\setminus A$  that is closed and discrete as subset of $M$,  any map $r:A\cup \Lambda\to\n$,  and  any group morphism 
$\pgot\colon H_1(M\setminus A,\z)\to \r^n$  with $\pgot|_{H_1(U\setminus A,\z)}=\Flux_X$,
 there is a full conformal minimal immersion $Y:M\setminus A\to\r^n$ satisfying the following conditions.
 \begin{enumerate}[{\rm (i)}]
 \item  $Y-X$ is harmonic at every point of $A$.
 \smallskip
 \item $Y-X$ vanishes at least to order $r(p)$ at each point $p\in A\cup\Lambda$.
 \smallskip
 \item $\Flux_Y=\pgot$.
 \smallskip
 \item $\|Y-X\|_U<\epsilon$ for any given $\epsilon>0$.
\end{enumerate}
Moreover, the immersion $Y$ can be chosen complete.
\end{theorem}

The proof of the theorem uses Theorem \ref{th:MT} and the following ad hoc technical lemma. The lemma is needed only to ensure the completeness of the conformal minimal immersion $Y\colon M\setminus A\to\r^n$ in Theorem \ref{th:M-Lgen}.
\begin{lemma}\label{lem:coorfij}
Let $M$ be an open Riemann surface and let $S=K\cup\Gamma\subset M$ be a Runge connected admissible subset. Assume that there is a component $K_0$ of $K$ that is a strong deformation retract of $M$. Let $E_0$ and $\Lambda$ be a pair of disjoint finite subsets of $\mathring K_0$, let $n\geq 3$ be an integer, and let   $X=(X_1,\ldots,X_n)\in \GCMI_\infty(S|E_0,\R^n)\cap \CMI_\infty(K|E_0,\R^n)$ be a map  such that
\begin{itemize}
\item $X|_{K_0}$ is full (see Definition \ref{def:full-CMI}).
\smallskip
\item $X_j$ extends  to $M\setminus E_0$ as a harmonic function,  $j\in \{3,\ldots,n\}$, and 
\smallskip
\item $\partial X_1^2+\partial X_2^2$ 
vanishes nowhere on $\Gamma$.
\end{itemize} 
 For any number $\epsilon>0$, any integer $r\geq 0$, and any smoothly bounded Runge compact domain $W\subset M$ with $S\subset \mathring W$,  there is a conformal minimal immersion $Y=(Y_1,\ldots,Y_n)\in \CMI_\infty(W|E_0,\R^n)$ satisfying the following conditions.
\begin{enumerate}[{\rm (i)}]
\item $Y$ is full.
\smallskip
 \item $Y-X$ extends to $S$ as a continuous map and $\|Y-X\|_S<\epsilon$.
\smallskip
\item $Y-X$ vanishes at least to order $r$ at every point of $\Lambda\cup E_0$.
\smallskip
\item $\Flux_Y=\Flux_X$.
\smallskip
\item $Y_j=X_j$ for all $j\in \{3,\ldots,n\}$.
\end{enumerate}
\end{lemma}


\begin{proof}
Fix a point $p_0\in \mathring K_0$ and take a skeleton $\{C_1,\ldots,C_l\}$, $l=m+\dim H_1(K_0,\z)$, of $K_0$ based at $(p_0,E_0\cup \Lambda)$ such that $C_i\cap C_j=\{p_0\}$ for all $i\neq j\in \{1,\ldots,l\}$ and $C=\bigcup_{j=1}^l C_j\subset \mathring K_0$  is a strong deformation retract of $K_0$, and hence of $S$ and $M$; see Definition \ref{de:skeleton} and recall that $K_0$ is a connected compact domain that is a strong deformation retract of $M$. 

 Let $\theta$ be a holomorphic $1$-form on $M$ vanishing nowhere. Write $u=\Xi(2\partial X/\theta)=(u_1,\ldots,u_n)\in \Ocal_\infty(K|E_0,\Sgot^{n-1}_*)\cap \Acal_\infty(S|E_0,\Sgot^{n-1}_*)$, where $\Xi$ is the map \eqref{eq:Thetabiho}. Notice that $u_j\in \Ocal_\infty(M|E_0)$, $j=3,\ldots,n$, and 
 \begin{equation}\label{eq:u1-u2}
 \text{$u_1, u_2$ \quad vanish nowhere on $\Gamma$; }
 \end{equation}
 recall that  $\partial X_1^2+\partial X_2^2$ vanishes nowhere on $\Gamma$. Consider the divisors on $\mathring K_0\subset M$
\begin{equation}\label{eq:Delta-ML}
	\Delta=\prod_{p\in E_0\cup \Lambda} p^r \text{\quad and \quad} \Delta_0= \Delta \big(\prod_{i=1}^n [u_i]_\infty\big),
\end{equation}
where $r$ is the number given in the statement of Lemma \ref{lem:coorfij}.  
 
Let $\epsilon_1>0$ to be specified later. 

The classical Weierstrass theorem \cite{Florack1948} gives $\varphi\in \Ocal_\infty(M|E_0)$ with $[\varphi]= [u_1]\in \Div(K)$, and hence vanishing nowhere on $M\setminus \supp ([u_1]_0)$; take into account \eqref{eq:u1-u2}. Since $u_1/\varphi\colon S\to\cp^1$ assumes values in $\c\setminus \{0\}$, which is an Oka manifold, there is $v_1\in \Ocal(M,\c\setminus \{0\})$ such that $ v_1\approx u_1/\varphi$ on $S$ and  $[v_1-u_1/\varphi]\geq  [u_1]_\infty  \Delta Z\in \Div(K)$,
where 
\[ 
	Z=\prod_{i=1}^n [u_i]^2_0 [u_i ]^2_\infty;
\] 
 see Forstneri\v c \cite[Theorems 3.8.1 and 5.4.4]{Forstneric2017}. Setting $\hat u_1:= v_1\varphi\in \Ocal_\infty(M|E_0)$  and choosing the approximation   close enough, then
\begin{equation}\label{eq:u1sust}
 \text{$\|\hat u_1-u_1\|_{S}<\epsilon_1$,\quad $[\hat u_1 -u_1 ]\geq  \Delta Z$,\quad and \quad$[\hat u_1]=[u_1]$}.
\end{equation} 
In particular, $\hat u_1$ vanishes nowhere on $M\setminus K$.
Since  $[\hat u_1]=[u_1]$, we have that $\hat u_2:=(\sum_{j=3}^n u_j^2)/\hat u_1$ lies in $\Ocal_\infty(M|E_0)$ and $[\hat u_2]=[u_2]$; it turns out that 
\[
 \hat u=(\hat u_1,\hat u_2, u_3,\ldots,u_n)\in \Ocal_\infty(M|E_0,\Sgot^{n-1}_*).
\]
 It follows that
\begin{equation}\label{eq:hatufinal}
\|\hat u-u\|_{S}< c_1 \epsilon_1 \quad \text{and}\quad \hat u_j-u_j\in \Ocal_{\Delta}(S)\quad \text{for all $j=1,2$,}
\end{equation}
where $c_1>0$ is a constant depending on $u$; use \eqref{eq:u1sust}.

Call $h=\big(( h_{i,j})_{i=1,\ldots,n}\big)_{j=1,\ldots,l} \in \big(\Ocal_{\Delta_0}(K_0)^n\big)^l$ 
  the map   given by  
 Lemma \ref{le:spray} applied to the data $M$, $K_0$,  the full map $u|_{K_0}\in \Ocal_\infty(K_0| E_0,\Sgot^{n-1}_*)$ and the $1$-form $\theta$. Choose a map
\[
	\sgot=\big((1+z),(1+z)^{-1},1\big)\in \Scal_2(\d)
\] 
(see Definition \ref{de:spindata}),
a number $\rho\in(0,1/ \| h\|_{K_0})$, and consider the associated multiplicative spray $ \Psi_\sgot\colon\rho\overline\b_{n,l} \to \Ocal_\Delta(K_0, u)$   in Lemma \ref{le:spray}-{\rm ({B}1)} (see   \eqref{eq:Omegaeta}); i.e,
\[
	\Psi_\sgot (\zeta)=\left((1+\zeta\cdot   h ) u_1, (1+\zeta\cdot   h )^{-1}  u_2,   (u_i)_{i=3,\ldots,n} \right).
\]
Here $\Delta$ and $\Delta_0$ are the divisors in \eqref{eq:Delta-ML} (cf.\ \eqref{eq:Delta*}). Fix $\epsilon_0>0$ and assume that $\rho$ is so small that Lemma \ref{le:spray}-{\rm ({B}2)},{\rm ({B}4)} are satisfied. By Theorem \ref{th:BSR}, we may assume that 
 \begin{equation}\label{eq:hlemafijo}
 	h\in \big(\Ocal_{\Delta_0}(M)^n\big)^l.
\end{equation}

Assume that $\rho>0$ is so small that $\rho \| h\|_{W}\le \rho \| h\|_{K_0}<1$, and consider the spray 
   $\hat \Psi_\sgot\colon\rho\overline\b_{n,l} \to \Ocal_\infty(W|E_0, \Sgot^{n-1}_*)$ given by
\begin{equation}\label{eq:fijarcoor}
  \hat \Psi_\sgot (\zeta)=\left((1+\zeta\cdot   h ) \hat u_1, (1+\zeta\cdot   h )^{-1} \hat u_2,   (u_i)_{i=3,\ldots,n} \right).
\end{equation} 
(Observe that the $n-2$ last components of $\hat \Psi_\sgot (\zeta)$ are those of $\Psi_\sgot (\zeta)$; here $W$ is the domain given in the statement of the lemma.)
 Note that, by  \eqref{eq:hatufinal} and \eqref{eq:hlemafijo}, 
\begin{equation}\label{eq:psivalues}
\text{$\hat \Psi_\sgot$ assumes values in $\Ocal_\Delta(K_0,  u)$}
\end{equation} 
 as well.
If  $\rho>0$ and  $\epsilon_1>0$   are chosen small enough,
 Lemma \ref{le:spray} and the Cauchy estimates  ensure that
\begin{enumerate}[\rm (a)]
\item $\hat\Psi_\sgot(\zeta)\colon W\setminus E_0 \to \Sgot_*^{n-1}\subset\c^n$ is full and  $\|\hat\Psi_\sgot(\zeta)-  u\|_{S}<  \epsilon_0$ for all   $ \zeta \in \rho\overline\b_{n,l}$; use Lemma \ref{le:spray}-{\rm ({B}2)}.
\smallskip
\item   $\Pcal^{1,2} \circ \hat\Psi_\sgot \colon \rho\overline\b_{n,l}  \to (\c^2)^l$   is a submersion  at $\zeta=0$ and $\Pcal^{1,2}  \big( \hat \Psi_\sgot(\zeta_0)\big)=0$ for some  $\zeta_0\in \rho \b_{n,l}$; see   \eqref{eq:Pcal-map*} and use Lemma \ref{le:spray}-{\rm ({B}4)}.
\end{enumerate}
Set $\hat f:=\Xi^{-1}(\hat \Psi_\sgot(\zeta_0))\in \Ocal_\infty(W|E_0,\Agot^{n-1}_*)$.

 Since  $S$ is path connected, $K_0\subset S$ is a strong deformation retract of $M$, and $\{C_1,\ldots,C_l\}$ is an skeleton of $K_0$ based at $(p_0,E_0\cup \Lambda)$, the definition of $\hat\Psi_\sgot$ in \eqref{eq:fijarcoor} and conditions {\rm (a)},   \eqref{eq:psivalues}, and {\rm (b)}  ensure that  
\[
Y\colon W\setminus E_0\to \r^n,\quad Y(p)=\Re\int_{p_0}^p \hat f \theta
\]
is a well defined, full complete conformal minimal immersion of FTC such that $Y_j=X_j$ for all $j=3,\ldots,n$, $\Flux_Y=\Flux_X$, and
$Y-X$ vanishes at least to order $r$ at every point of $\Lambda\cup E_0$. Finally, {\rm (a)} and the compactness of $S$ ensure that  $\|Y-X\|_{S}<\epsilon$ provided that $\epsilon_0$ is chosen small enough. This concludes the proof.
\end{proof}

\begin{proof}[Proof of Theorem \ref{th:M-Lgen}]
Since $U$ is locally connected, every compact set in $M$ intersects at most finitely many components of $U$. Therefore, there is a sequence of connected, smoothly bounded, Runge compact domains 

\begin{equation}\label{eq:M0M1}
	M_0\Subset M_1\Subset M_2\Subset \cdots\subset \bigcup_{j\in\z_+}M_j=M
\end{equation}
such that $M_0$ is a disc, $U\cap M_0=\varnothing$, and $U\cap bM_j=\varnothing$ for all $j\ge 1$.  Fix $p_0\in\mathring M_0$.

Take  $\epsilon>0$. Fix a nowhere vanishing holomorphic $1$-form $\theta$ on $M$. Also choose a full conformal minimal immersion $X_0\colon M_0\to\r^n$ and a number 
\begin{equation}\label{eq:epsilon0-ML}
	\epsilon_0\in \big(0,\frac12\min\left\{\epsilon,\delta_0\right\}\big),\quad \text{where $\delta_0= \min \{ |\partial X_0/\theta|(p)\colon p\in M_0\}>0$.}
\end{equation}
Furthermore, we choose $\epsilon_0>0$ so small that every conformal minimal immersion $Z\colon M_0\to\r^n$ with $\|Z-X_0\|_{M_0}<2\epsilon_0$ is full. 
We shall inductively construct a sequence of numbers $\epsilon_j>0$ and full immersions $X_j\in \CMI_\infty(M_j|A\cap M_j,\r^n)$ satisfying the following conditions for all $j\ge 1$.
\begin{enumerate}[\rm (1$_j$)]
\item $\max\big\{\|X_j-X_{j-1}\|_{M_{j-1}} \,,\, \|(\partial X_j-\partial X_{j-1})/\theta\|_{M_{j-1}}\big\}<\epsilon_{j-1}$.
\smallskip
\item $\|X_j-X\|_{U\cap M_j\setminus M_{j-1}}<\epsilon_{j-1}$.
\smallskip
\item $X_j-X$ extends harmonically to $\mathring U\cap M_j$.
\smallskip
\item $X_j-X$ vanishes at least to order $r(p)$ at each point $p\in (A\cup\Lambda)\cap M_j$.
\smallskip
\item $\Flux_{X_j}=\pgot|_{H_1(M_j\setminus A,\z)}$.
\smallskip
\item $\dist_{X_j}(p_0,bM_i)>i$ for all $i\in\{0,\ldots,j\}$.
\smallskip
\item $\epsilon_j<\frac12\min\{\epsilon_{j-1},\delta_j\}$, where $\delta_j=\min\big\{ |\partial X_j/\theta|(p)\colon p\in M_j\setminus A\big\}>0$.
\end{enumerate}

Assume that such a sequence exists. By properties \eqref{eq:M0M1}, {\rm (1$_j$)}, and {\rm (7$_j$)}, there is a limit map
\[
	Y=\lim_{j\to\infty} X_j\colon M\setminus A\to\r^n
\]
that is a conformal harmonic map satisfying
\begin{equation}\label{eq:Y-Xj-ML}
	\max\left\{\|Y-X_j\|_{M_j} \,,\,  \Big\|\frac{\partial Y-\partial X_j}{\theta}\Big\|_{M_j}\right\}<2\epsilon_j<\delta_j\quad \text{for all $j\ge1$}.
\end{equation}
In particular, $Y\colon M\setminus A\to\r^n$ is a full conformal minimal immersion. By \eqref{eq:epsilon0-ML}, {\rm (2$_j$)}, and \eqref{eq:Y-Xj-ML}, we have that $\|Y-X\|_{U\cap M_j\setminus M_{j-1}}<2\epsilon_{j-1}\le 2\epsilon_0<\epsilon$ for all $j\ge 1$. Since $U\cap M_0=\varnothing$, this and \eqref{eq:M0M1} imply condition {\rm (iv)}. It is clear that {\rm (3$_j$)}, {\rm (4$_j$)}, and {\rm (5$_j$)} ensure {\rm (i)}, {\rm (ii)}, and {\rm (iii)}.  Finally, properties {\rm (ii)} and {\rm (6$_j$)} and the completeness of $X$, guarantee that $Y$ is complete. Thus, $Y$ satisfies the conclusion of the theorem.

Let us now explain the induction. The basis is given by the already fixed number $\epsilon_0>0$ and conformal minimal immersion $X_0\colon M_0\to\r^n$. Note that, since $A\cup\Lambda\subset U$ and $U\cap M_0=\varnothing$, we obviously have that $X_0\in\CMI_\infty(M_0|A\cap M_0,\r^n)$ (see \eqref{eq:CCMI}) and conditions {\rm (3$_0$)} and {\rm (4$_0$)} are satisfied. Moreover, since $M_0$ is simply connected and $p_0\in\mathring M_0$, conditions {\rm (5$_0$)} and {\rm (6$_0$)} hold true as well. Finally, conditions {\rm (1$_0$)}, {\rm (2$_0$)}, and {\rm (7$_0$)} are void. For the inductive step, assume that for some $j\in\n$ we have numbers $\epsilon_0,\ldots,\epsilon_{j-1}$ and full immersions $X_0,\ldots,X_{j-1}$ satisfying the required conditions for all $i\in\{0,\ldots,j-1\}$, and let us provide $\epsilon_j$ and $X_j$. 

Choose a connected, smoothly bounded, Runge compact domain $M_j'$ in $M$ such that $M_j\subset M_j'$ and $M_j$ is a strong deformation retract of $M_j'$. Let $\Sigma$ be a compact Riemann surface (without boundary) such that $M_j'$ is a smoothly bounded compact domain in $\Sigma$, and let $E\subset \Sigma\setminus M_j'$ be a finite set such that $M_j'$ is Runge in $\Sigma\setminus E$. Recall that $U\cap (bM_{j-1}\cup bM_j)=\varnothing$ and note that $M_{j-1}\cup (U\cap M_j)$ is Runge and admissible in $\Sigma\setminus E$. By Theorem \ref{th:MT}, there is a full complete conformal minimal immersion $Y_j\colon \Sigma\setminus (E\cup (A\cap M_j))\to\r^n$ of finite total curvature such that $Y_j|_{M_j}\in \CMI_\infty(M_j|A\cap M_j,\r^n)$, satisfies conditions {\rm (1$_j$)}--{\rm (5$_j$)}, and satisfies condition {\rm (6$_j$)} for all indices $i\in\{0,\ldots,j-1\}$, but need not satisfy $\dist_{Y_j}(p_0,bM_j)>j$. We now perturb $Y_j$ near $bM_j$ in order to ensure that inequality.
For, choose a connected, smoothly bounded, Runge compact domain $M_j''$ in $M$ such that $M_{j-1}\cup (U\cap M_j)\Subset M_j''\Subset M_j$ and $M_j''$ is a strong deformation retract of $M_j$. Write $Y_j=(Y_{j,1},\ldots,Y_{j,n})$. Since $Y_j$ is full, we have that $Y_{j,n}$ is nonconstant. Choose a Runge compact set $K\subset \mathring M_j\setminus M_j''$ in $M$ such that $K$ is a finite union of smoothly bounded compact discs and
\begin{equation}\label{eq:JX}
	\int_\gamma |\partial Y_{j,n}|>1
\end{equation}
for all paths $\gamma\colon [0,1]\to M\setminus K$ with   $\gamma(0)\in M_j''$ and $\gamma(1)\in M\setminus \mathring M_j$. Existence of such a set is well known; we refer e.g.\ to \cite{JorgeXavier1980AM,AlarconFernandezLopez2013CVPDE,AlarconCastro-Infantes2019APDE}. Fix a number $T>0$ so large that
\begin{equation}\label{eq:JX2}
	\min\{|Y_{j,1}(p)+T|\colon p\in K\}>\|Y_{j,1}\|_{bM_j''}+2.
\end{equation} 
Let $\Gamma$ be a finite family of pairwise disjoint smooth Jordan arcs in $\mathring M_j$ such that $S=(M_j''\cup K)\cup \Gamma$ is a connected admissible Runge subset of $M$ that is a strong deformation retract of $M_j$ and  $\partial Y_{j,1}^2+\partial Y_{j,2}^2$ vanishes nowhere on $\Gamma\cap (M_j''\cup K)$.  Consider any  immersion
\[
Y_j'=(Y_{j,1}',\ldots,Y_{j,n}') \in\GCMI_\infty(S| A\cap M_j,\r^n)\cap \CMI_\infty(M_j''\cup K| A\cap M_j,\r^n)
 \]
 such that $Y_j'=Y_j$ on $M_j''$ and $Y_j'=Y_j+(T,0,\ldots,0)$ on $K$; it turns out that $(\partial Y'_{j,1})^2+(\partial Y'_{j,2})^2$ vanishes nowhere on $\Gamma$. Since the immersion $Y_j|_{M_j}$ satisfies conditions {\rm (1$_j$)}--{\rm (5$_j$)} and condition {\rm (6$_j$)} for all the indices $i\in\{0,\ldots,j-1\}$, Lemma \ref{lem:coorfij} furnishes us for any small enough $\epsilon'>0$ with a full immersion $X_j=(X_{j,1},\ldots,X_{j,n})\in \CMI_\infty(M_j|A\cap M_j,\r^n)$ satisfying the same conditions and, in addition,
\begin{enumerate}[\rm (a)]
\item $X_{j,n}=Y_{j,n}$ and
\smallskip
\item $\|X_{j,1}-Y_{j,1}'\|_{M_j''\cup K}<\epsilon'$.
\end{enumerate}
Let us see that $X_j$ satisfies {\rm (6$_j$)}. For, since $\dist_{X_j}(p_0,bM_{j-1})>j-1$ and $p_0\in \mathring M_{j-1}\Subset \mathring M_j''$, it suffices to check that $\int_\gamma |\partial X_j|>1$ for all paths $\gamma\colon[0,1]\to M_j\setminus\mathring M_j''$ with $\gamma(0)\in bM_j''$ and $\gamma(1)\in bM_j$; recall that $2|\partial X_j|^2$ is the metric induced on $M_j$ by the Euclidean metric in $\r^n$ via the immersion $X_j$ (see \eqref{eq:ds2}). Let $\gamma$ be such a path. If $\gamma([0,1])\cap K=\varnothing$, then 
\[
	\int_\gamma |\partial X_j|\ge \int_\gamma |\partial X_{j,n}|\stackrel{\rm (a)}{=} \int_\gamma |\partial Y_{j,n}|\stackrel{\eqref{eq:JX}}{>}1.
\]
If, on the contrary, $\gamma([0,1])\cap K\neq \varnothing$, then for any point $p\in \gamma([0,1])\cap K$ we have
\begin{eqnarray*}
	\int_\gamma |\partial X_j| &\ge & |X_j(p)-X_j(\gamma(0))|
	\\
	& \ge & |X_{j,1}(p)|-|X_{j,1}(\gamma(0))|
	\\
	& \stackrel{\rm (b)}{>} & |Y_{j,1}(p)+T|-|Y_{j,1}(\gamma(0))|-2\epsilon' \; \stackrel{\eqref{eq:JX2}}{>} \; 2-2\epsilon'>1,
\end{eqnarray*}
where for the last inequality we assume that $\epsilon'<1/2$. This shows {\rm (6$_j$)}.

Finally, choose any $\epsilon_j>0$ so small that {\rm (7$_j$)} is satisfied for this $X_j$. This ensures the inductive step and completes the proof of the theorem.  
\end{proof}
%
%
\begin{remark}\label{rem:C1}
The approximations in Theorems \ref{th:Merg}, \ref{th:MT}, and \ref{th:M-Lgen} take place in the natural $\Ccal^1$ topology for (generalized) conformal minimal immersions, despite it is not mentioned in their statements. Indeed, just observe that convergence of the Weierstrass data is ensured in the proofs. Furthermore, in view of the recent result by Forn\ae ss, Forstneri\v c, and Wold \cite[Theorem 16]{FornaessForstnericWold2018} on Mergelyan approximation in the $\Ccal^r$ topology on admissible sets, it seems that the results in this paper can be extended by guaranteeing approximation of this class.
\end{remark}


\subsection*{Acknowledgements}
The authors were partially supported by the State Research Agency (SRA) and European Regional Development Fund (ERDF) via the grant no.\ MTM2017-89677-P, MICINN, Spain, and the Junta de Andaluc\'ia projects no. P18-FR-4049 and A-FQM-139-UGR18 (FEDER). 

We are grateful to the referee for the careful reading of the paper and the valuable comments that improved the presentation.




\medskip
\noindent Antonio Alarc\'{o}n

\noindent Departamento de Geometr\'{\i}a y Topolog\'{\i}a e Instituto de Matem\'aticas (IEMath-GR), Universidad de Granada, Campus de Fuentenueva s/n, E--18071 Granada, Spain.

\noindent  e-mail: {\tt alarcon@ugr.es}

\bigskip

\noindent Francisco J. L\'opez

\noindent Departamento de Geometr\'{\i}a y Topolog\'{\i}a e Instituto de Matem\'aticas (IEMath-GR), Universidad de Granada, Campus de Fuentenueva s/n, E--18071 Granada, Spain.

\noindent  e-mail: {\tt fjlopez@ugr.es}

\end{document}